\documentclass[a4paper,reqno,11pt,oneside]{amsart}
\usepackage{amsfonts,amssymb,latexsym,xspace,epsfig,color}
\usepackage{amsmath,stmaryrd,xy}
\usepackage{a4wide}
\usepackage[small,bf]{caption}
\usepackage{color}
\usepackage{enumitem}
\usepackage{float}
\usepackage{graphicx}
\usepackage{mathtools}
\numberwithin{figure}{section}
\usepackage{subfigure}
\usepackage{cite}

\usepackage [latin1]{inputenc}

\usepackage{tikz}
\usetikzlibrary{arrows,snakes,backgrounds,shapes,automata}
\usetikzlibrary{positioning}
\usepackage{fontawesome}

\tikzset{
    man/.pic={%
        \fill [rounded corners=1.5] (0,0.4) -- (0,0.8) -- (0.4,0.8) -- (0.4,0.4) --
            (0.325,0.4) -- (0.325,0.7) -- (0.3,0.7) -- (0.3,0) -- (0.225,0) --
            (0.225,0.4) -- (0.175,0.4) -- (0.175,0) -- (0.1,0) -- (0.1,0.7) --
            (0.075,0.7) -- (0.075,0.4) -- cycle;
        \fill (0.2,0.9) circle (0.1);
        \coordinate (-head) at (0.2,1);
        \coordinate (-foot) at (0.2,0);
    }
}

\theoremstyle{plain}
\newtheorem{theorem}{Theorem}[section]                                          
\newtheorem{proposition}[theorem]{Proposition}                          
\newtheorem{lemma}[theorem]{Lemma}

\theoremstyle{definition}
\newtheorem{definition}[theorem]{Definition}
\newtheorem{remark}{Remark}[section]

\numberwithin{equation}{section}
\allowdisplaybreaks[1]

\newlist{lteo}{enumerate}{1}
\setlist[lteo,1]{wide,itemsep={6pt plus 4pt},font=\bfseries,label={(\arabic*)}}

\newlist{iteo}{enumerate}{1}
\setlist[iteo,1]{wide,itemsep=2ex,label={\upshape(\roman*)}}

\bibdata{prob}
\bibstyle{alpha} 

\newcommand{\dNormal}{\mathcal{N}}

\newcommand{\Cov}{\text{Cov}}

\newcommand{\xf}{x_{\infty}}
\newcommand{\yuf}{y_{1,\infty}}
\newcommand{\ydf}{y_{2,\infty}}
\newcommand{\tf}{\tau_{\infty}}
\newcommand{\df}{\delta_{\infty}}
\newcommand{\lf}{\Lambda_{\infty}}
\newcommand{\GV}{\mathcal{V}}

\newcommand{\vetord}{\ensuremath{(x, y_1, y)}}
\newcommand{\vetort}{\ensuremath{(x, y_1, y_2, y)}}
\newcommand{\nzeros}[2]{\ensuremath{Z(#1, #2)}}
\newcommand{\intazu}{\ensuremath{(0, 1)}}
\newcommand{\intfzu}{\ensuremath{(0, 1]}}
\newcommand{\intazx}{\ensuremath{(0, x_0)}}
\newcommand{\intfzx}{\ensuremath{(0, x_0]}}
\newcommand{\intazi}{\ensuremath{(0, \infty)}}
\newcommand{\intfzi}{\ensuremath{[0, \infty)}}

\makeatletter
\@namedef{subjclassname@2020}{%
  \textup{2020} Mathematics Subject Classification}
\makeatother

\title[The role of multiple repetitions on the size of a rumor]{The role of multiple repetitions on the size of a rumor}

\author{Alejandra Rada, Cristian F. Coletti, Elcio Lebensztayn and Pablo M. Rodriguez}

\date{}

\address{
\newline
Alejandra Rada and Cristian F. Coletti
\newline
UFABC - Centro de Matem\'atica, Computa\c{c}\~ao e Cogni\c{c}\~ao
\newline
Avenida dos Estados, 5001- Bangu - Santo André - São Paulo, Brasil.
\newline
e-mails:  cristian.coletti@ufabc.edu.br, alejandra.rada@ufabc.edu.br
\newline
\newline
Elcio Lebensztayn
\newline
Instituto de Matem\'atica, Estat\'istica e Computa\c{c}\~ao Cient\'ifica, Universidade Estadual de Campinas
\newline
Rua S\'ergio Buarque de Holanda 651, 13083-859, Campinas, SP, Brasil.
\newline
e-mail: lebensztayn@ime.unicamp.br 
\newline
\newline
Pablo M. Rodriguez
\newline
Centro de Ci\^encias Exatas e da Natureza, Universidade Federal de Pernambuco
\newline
Av. Prof. Moraes Rego, 1235. Cidade Universit\'aria, CEP 50670-901, Recife - PE, Brasil.
\newline
e-mail:  pablo@de.ufpe.br
}

\thanks{}

\begin{document}

\maketitle

\begin{abstract}
We propose a mathematical model to measure how multiple repetitions may influence the ultimate proportion of the population never hearing a rumor while it is being spread. The model is a multi-dimensional continuous-time Markov chain that can be seen as a generalization of the Maki--Thompson model for the propagation of a rumor within a homogeneously mixing population. In the well-known basic model, the population is made up of ``spreaders'', ``ignorants'' and ``stiflers'', and any spreader attempts to transmit the rumor to the other individuals via directed contacts. When the contacted individual is an ignorant, it becomes a spreader, while in the other two cases the initiating spreader becomes a stifler. The process in a finite population eventually reaches an equilibrium situation, where individuals are either stiflers or ignorants. We generalize the model by assuming that each ignorant becomes a spreader only after hearing the rumor a predetermined number of times. We identify and analyze a suitable limiting dynamical system of the model, and prove limit theorems that characterize the ultimate proportion of individuals in the different classes of the population. 
\end{abstract}

\maketitle

\section{Introduction}

The similarity between the spreading mechanisms of rumors and diseases was noted in the earlier 1960s by various writers but questioned by Daley and Kendall \cite{DKNature,DK}. They proposed a simple mathematical model for rumor transmission which appeared as an alternative to the well-known susceptible-infected-removed epidemic model (SIR). 
While a SIR model assumes that a population is subdivided into susceptible, infected and removed individuals, in Daley and Kendall's model, the individuals are classified as ignorants, spreaders or stiflers. At this point, the similarity between both classifications becomes obvious: ignorants in rumors are similar to susceptible in epidemics, spreaders are similar to infected cases, and stiflers correspond to removals. 
The main difference between these two processes is in the stopping mechanisms. While in the propagation of a disease an infected individual becomes removed only after a random time, which is in general independent of what happens with other individuals, in the transmission of a rumor a spreader stops spreading a piece of information right after getting involved in a contact with another individual who already knows it. 
The last situation describes the essence of a rumor model: it incorporates the event of losing  interest in propagating the rumor as a result of learning that it is already known by the other person in the meeting. In other words, in a rumor process, a spreader becomes a stifler at a rate which depends on the number of non-ignorant individuals in the population. This communication and the findings summarized in \cite{DK} gave rise to the theory of mathematical models for rumor transmission. 

The next step in the development of this theory was the formulation of a simplification in the Daley--Kendall model, due to Maki and Thompson, in \cite{maki-1973}. The Maki--Thompson model also describes the spreading of a rumor on a closed homogeneously mixed population, subdivided into the three classes of individuals mentioned before: ignorants (those not aware of the rumor), spreaders (who are spreading it), and stiflers (who know the rumor but have ceased communicating it after meeting somebody who has already heard it). 
But it assumes that the rumor spreads by {\it directed} contacts of spreaders with other individuals.
For $t\geq 0$, the number of ignorants, spreaders and stiflers at time $t$ is denoted by $X^{N}(t)$, $Y^{N}(t)$ and $Z^{N}(t)$, respectively. 
The model starts with $X^{N}(0) = N-1$, $Y^{N}(0) = 1$, $Z^{N}(0) = 0$, and $X^{N}(t) + Y^{N}(t) + Z^{N}(t) = N $ for all~$t$. The Maki--Thompson model is the continuous-time Markov chain $\{(X^{N}(t), Y^{N}(t))\}_{t \in [0,\infty)}$ which evolves according to the following transitions and rates
\begin{equation}\label{eq:transMT}
\begin{array}{cc}
\text{transition} \quad &\text{rate} \\[0.1cm]
(-1, 1) \quad &X Y, \\[0.1cm]
(0, -1) \quad &Y (N -1 - X).
\end{array}
\end{equation}

This means that if the process is in state $(i,j)$ at time $t$, then the probabilities that it jumps to states $(i-1,j+1)$ or $(i,j-1)$ at time $t+h$ are, respectively, $i\,j\,h + o(h)$ and $j (N - 1 - i)\,h + o(h)$, where $o(h)$ represents a function such that $\lim_{h\to 0}o(h)/h =0$.
This describes exactly the situation in which individuals interact by contacts initiated by the spreaders: the two possible transitions in \eqref{eq:transMT} correspond to spreader-ignorant, and spreader-(spreader or  stifler) interactions. 
In the first case, the spreader tells the rumor to the ignorant, who becomes a spreader. The other transition represents the transformation of a spreader into a stifler after initiating a meeting with a non-ignorant. 
That is, the last event describes the loss of interest in propagating the rumor derived from learning that it is already known by the other individual in the meeting. 
As we have mentioned, this is the main difference between rumor and epidemic models like SIR; see Figure \ref{FIG:transitionsMTvsSIR}. Note that, in the Maki--Thompson model, when a spreader contacts another spreader, only the initiating one becomes a stifler.

\begin{figure}[ht]
\begin{center}
\begin{tikzpicture}



\draw[gray!70] (-6,-2) to (6,-2);


\node at (0,-0.5) {


\begin{tikzpicture}[scale=1.3,every node/.style={scale=1.3}]

\pic[gray!70] at (-1.3,0) (myman) {man};
\node[below=1mm of myman-foot] {{\tiny ignorant/susceptible}};
\draw [fill=cyan!70,cyan!70] (-1.2,0.39) rectangle (-1,0.793);


\pic[gray!70] at (2,0) (myman) {man};
\node[below=1mm of myman-foot] {{\tiny spreader/infected}};
\draw [fill=red,red] (2.1,0.39) rectangle (2.3,0.793);


\pic[gray!70] at (5.3,0) (myman) {man};
\node[below=1mm of myman-foot] {{\tiny stifler/removed}};
\draw [fill=blue!30!black,blue!30!black] (5.4,0.39) rectangle (5.6,0.793);
\end{tikzpicture}

};

\node at (0,-2.95) {$vs$};

\node at (-3,-6) {
\begin{tikzpicture}

\node at (1.92,6.5) {MT rumor model};


\pic[gray!70] at (0,4) (myman) {man};
\draw [fill=red,red] (0.1,4.39) rectangle (0.3,4.793);
\pic[gray!70] at (0.5,4) (myman) {man};
\draw [fill=cyan!70,cyan!70] (0.6,4.39) rectangle (0.8,4.793);

\pic[gray!70] at (3,4) (myman) {man};
\draw [fill=red,red] (3.1,4.39) rectangle (3.3,4.793);
\pic[gray!70] at (3.5,4) (myman) {man};
\draw [fill=red,red] (3.6,4.39) rectangle (3.8,4.793);

\draw[->,gray!70,thick] (1.35,4.5) -- (2.5,4.5);
\draw [->,dashed] (0.8,5.15) to [bend left=45] (3.6,5.15);


\pic[gray!70] at (0,2) (myman) {man};
\draw [fill=red,red] (0.1,2.39) rectangle (0.3,2.793);
\pic[gray!70] at (0.5,2) (myman) {man};
\draw [fill=red,red] (0.6,2.39) rectangle (0.8,2.793);

\pic[gray!70] at (3,2) (myman) {man};
\draw [fill=blue!30!black,blue!30!black] (3.1,2.39) rectangle (3.3,2.793);
\pic[gray!70] at (3.5,2) (myman) {man};
\draw [fill=red,red] (3.6,2.39) rectangle (3.8,2.793);

\draw[->,gray!70,thick] (1.35,2.5) -- (2.5,2.5);
\draw [->,dashed] (0.3,3.15) to [bend left=45] (3.1,3.15);


\pic[gray!70] at (0,0) (myman) {man};
\draw [fill=red,red] (0.1,0.39) rectangle (0.3,0.793);
\pic[gray!70] at (0.5,0) (myman) {man};
\draw [fill=blue!30!black,blue!30!black] (0.6,0.39) rectangle (0.8,0.793);

\pic[gray!70] at (3,0) (myman) {man};
\draw [fill=blue!30!black,blue!30!black] (3.1,0.39) rectangle (3.3,0.793);
\pic[gray!70] at (3.5,0) (myman) {man};
\draw [fill=blue!30!black,blue!30!black] (3.6,0.39) rectangle (3.8,0.793);

\draw[->,gray!70,thick] (1.35,0.5) -- (2.5,0.5);
\draw [->,dashed] (0.3,1.15) to [bend left=45] (3.1,1.15);
\end{tikzpicture}};

\node at (3,-5) {
\begin{tikzpicture}


\node at (1.92,6.5) {SIR epidemic model};

\pic[gray!70] at (0,4) (myman) {man};
\draw [fill=red,red] (0.1,4.39) rectangle (0.3,4.793);
\pic[gray!70] at (0.5,4) (myman) {man};
\draw [fill=cyan!70,cyan!70] (0.6,4.39) rectangle (0.8,4.793);

\pic[gray!70] at (3,4) (myman) {man};
\draw [fill=red,red] (3.1,4.39) rectangle (3.3,4.793);
\pic[gray!70] at (3.5,4) (myman) {man};
\draw [fill=red,red] (3.6,4.39) rectangle (3.8,4.793);

\draw[->,gray!70,thick] (1.35,4.5) -- (2.5,4.5);
\draw [->,dashed] (0.8,5.15) to [bend left=45] (3.6,5.15);


\pic[gray!70] at (0.25,2) (myman) {man};
\draw [fill=red,red] (0.35,2.39) rectangle (0.55,2.793);

\pic[gray!70] at (3.25,2) (myman) {man};
\draw [fill=blue!30!black,blue!30!black] (3.35,2.39) rectangle (3.55,2.793);

\draw[->,gray!70,thick] (1.35,2.5) -- (2.5,2.5);

\end{tikzpicture}};

\end{tikzpicture}

\end{center}
\caption{Illustration of all possible transitions between the different classes of individuals in the Maki--Thompson rumor model and the SIR epidemic model. The population is subdivided into ignorants, spreaders and stiflers for rumors and in susceptible, infected and removed for diseases. The different transitions in both models involve interactions between these classes of individuals. When the change is a consequence of the interaction between two individuals, we present the result by a dashed arrow.}\label{FIG:transitionsMTvsSIR}
\end{figure}
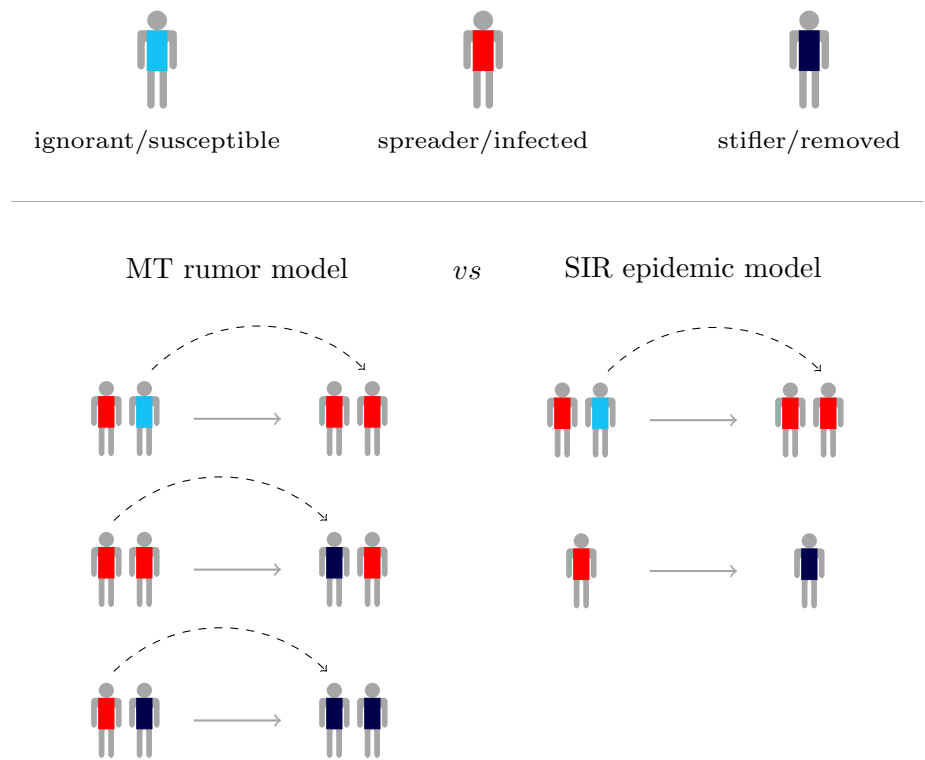

The deterministic version for the Maki--Thompson model is nothing more than the dynamical system given by:

\begin{equation}\label{eq:system}
\left\{
\begin{array}{l}
x^{\prime}(t) = - x(t) y(t),\\[0.2cm]
y^{\prime}(t) = (2x(t)-1)\,y(t),\\[0.2cm]
x(0)=1, y(0)=0.
\end{array}\right.
\end{equation}
The solution of \eqref{eq:system} represents an approximation, for sufficiently large $N$, of a scaled version of the entire trajectories of the Maki--Thompson model. Moreover, $x(t)$ and $y(t)$ may be seen as good approximations for the proportion of ignorants and spreaders, respectively, at time $t$, for $t>0$ and $N$ large enough. This justifies the fact that many results may be obtained directly from a suitable analysis of this limit dynamical system. However, a deterministic approach is not enough for describing the behavior of the random fluctuations between the solution of \eqref{eq:system} and the random trajectories coming from the original stochastic process. Although for the sake of simplicity we do not describe the Daley--Kendall model here, it is worth pointing out that the Daley--Kendall model has the same limit dynamical system \eqref{eq:system} as its deterministic version. This is the reason why the Maki--Thompson model has been developed as an alternative for the Daley--Kendall model. 
An extensive treatment of rumor models can be found in \cite[Chapter~5]{DG}.

A fundamental aspect of stochastic rumor models in finite populations is that the process eventually ends, so the main interest is to get information about the remaining proportion of people who never hear the rumor. 
This is one way of having a measure of, say, the size of the rumor. The first rigorous results in this direction are limit theorems for the remaining proportion of ignorants when the process ends, as the population size grows to $\infty$. 
It has been proved that, for both the Daley--Kendall and the Maki--Thompson models, such final  proportion of ignorants equals approximately $20\%$, see \cite{Sudbury,Watson}. 
The arguments used in these works rely on a martingale approach to guarantee the approximation of the random trajectories by the limit dynamical system given in~\eqref{eq:system}. 
For recent results of the theory of stochastic rumor models in homogeneously mixed populations, we refer the reader to \cite{Grejo/Rodriguez/2019,Lebensztayn-JMAA2015,lebensztayn/machado/rodriguez/2011a,lebensztayn/machado/rodriguez/2011b,EP}, and the references given there. 
Results for the Maki--Thompson model, assuming that the population is not necessarily homogeneous nor totally mixing, can be found in \cite{EAP,arruda,raey,speroto,MNP-PRE2004,moreno-PhysA2007}, and the references therein. 

In this work, we formulate a generalization for the Maki--Thompson model, by assuming that each ignorant becomes a spreader only after hearing the rumor a certain number of times. This assumption is motivated by recent experimental research, where it has been found that repeated presentation of uncertain statements increases validity judgments of those statements, see \cite{difonzo}. 
As far as we know, no rigorous results exist for this type of model describing multi-repetition in the propagation of a rumor. 
Our model evolves as a multidimensional Markov chain; we establish limit theorems by proving its convergence, in a way to be specified later, to a dynamical system suitably defined. 
Our approach relies on the application of convergence results coming from the theory of density dependent Markov chains. The advantage of using this theory lies in the fact that we can obtain results that not only localize the asymptotic proportion of individuals in the different classes of the population, but also provide a description of the random fluctuations between these values and those from the original stochastic process.

The paper is organized as follows. In Section \ref{S:model}, we define our model formally and state the main results. 
Our approach requires the identification of a limiting dynamical system whose solution is partially expressed in terms of a transcendental function $f$. 
The role of this function is the localization of the asymptotic proportion of remaining people who never hear the rumor. 
Also in Section \ref{S:model}, we present a brief exposition of some properties of $f$ and its zeros, which might be of independent interest. Our results are proved in Sections \ref{S:proofLLN}, \ref{S:Proof f} and \ref{S:Proofs CLT}.

\section{Model and main results}\label{S:model}

\subsection{Definition of the model}

Consider a closed homogeneously mixed population with $N$ individuals. 
For a fixed $k\in \mathbb{N}$, let us consider the \textit{$k$-spreading Maki--Thompson model}, which has the same basic rules as the classical Maki--Thompson model, except that  an ignorant individual becomes a spreader only after being involved in $k$ interactions with spreaders.
The population is subdivided into $k+2$ classes of individuals: ignorants, $i$-aware individuals for $i\in\{1,\ldots,k-1\}$ (who already know the rumor, but do not want to spread it), spreaders, and stiflers.
An $i$-aware individual has heard the rumor exactly $i$ times.
For $t\geq 0$, we denote the number of individuals at each one of these classes at time $t$ by $X^{N}(t)$, $Y^{N}_i(t)$, $Y^{N}(t)$ and $Z^{N}(t)$, respectively. Thus, the $k$-spreading Maki--Thompson model is the $(k+1)$-dimensional continuous-time Markov chain $\left\{\left(X^{N}(t),Y^{N}_1(t), \ldots,Y^{N}_{k-1}(t), Y^{N}(t)\right)\right\}_{t\in[0,\infty)}$ with the following increments and rates:
\begin{equation}\label{EQ:transitions}
{\allowdisplaybreaks
\begin{array}{ccc}
\text{increment} \quad &\text{rate} & \\[0.2cm]
- {\bf e}_1 + {\bf e}_2 \quad &X Y,& \\[0.2cm]
- {\bf e}_{i+1} + {\bf e}_{i+2} \quad &Y_i Y,& i\in\{1,2,\ldots,k-1\} \\[0.2cm]
- {\bf e}_{k+1} \quad &\left(N-1-X-\displaystyle\sum_{i=1}^{k-1}Y_i \right) Y,&
\end{array}}%
\end{equation}
where $\{{\bf e}_1, {\bf e}_2,\ldots,{\bf e}_{k+1}\}$ is the natural basis of the $(k+1)$-dimensional Euclidean space. In words, the transitions in \eqref{EQ:transitions} represent the observed result after different interactions between individuals (see Figure \ref{FIG:transitions}). The first transition is a consequence of an interaction between a spreader and an ignorant, which implies the ignorant becoming a $1$-aware individual. 
For each $i\in\{1,\ldots,k-1\}$, the second transition comes from the interplay between a spreader and an $i$-aware individual; in this case the $i$-aware individual becomes an $(i+1)$-aware one. 
Finally, the increment and the rate in the third line of~\eqref{EQ:transitions} correspond to the situation when a spreader contacts another spreader or a stifler individual, which produces a new stifler in the population. 

We assume that the following limits exist:
\begin{equation}\label{eq:limits_0}
\lim_{N\to \infty} \frac{X^N(0)}{N} =:x_{0},\;\;\;\; \lim_{N\to \infty} \frac{Y_{i}^N(0)}{N} =: y_{i,0}, \text{ for }i\in\{1,\ldots,k-1\},\;\;\;\;\lim_{N\to \infty} \frac{Y^N(0)}{N} =: y_{0}, 
\end{equation}	
with $x_0\in (0,1]$ and $y_{i,0}, y_0 \in [0,1]$. In addition, we let 
\[ z_{0}:= \lim_{N\to \infty} \frac{Z^N(0)}{N} = 1- \left(x_0 + \sum_{i=1}^{k-1} y_{i,0}+ y_0\right). \]
We call the set of limits in \eqref{eq:limits_0} the \textit{initial configuration of the process}; we refer to the case when $x_0=1$ as the \textit{standard initial configuration}. 
Notice that $X^N(t)+\sum_{i=1}^{k-1}Y_{i}^N(t)+Y^N(t)+Z^N(t)=N$ for any $t\geq 0$, i.e., we consider a closed finite population of size $N$. We point out that the classical Maki--Thompson model is obtained by considering $k =1$ and the standard initial configuration (the basic version assumes $X^N(0)=N-1$ and $Y^N(0)=1$).

\begin{figure}[ht]
\begin{center}
\begin{tikzpicture}



\draw[gray!70] (-1.7,1.9) to (-1.7,7);


\node at (-4,4.5) {


\begin{tikzpicture}[scale=1.3,every node/.style={scale=1.3}]
\pic[gray!70] at (0,6) (myman) {man};
\node[below=0mm of myman-foot] {{\footnotesize ignorant}};
\draw [fill=cyan!70,cyan!70] (0.1,6.39) rectangle (.3,6.793);
\node at (0.5,7.1) {{\footnotesize \faQuestion}};


\pic[gray!70] at (0,4) (myman) {man};
\node[below=1mm of myman-foot] {{\footnotesize $i$-aware}};
\draw [fill=blue!20,blue!20] (0.1,4.39) rectangle (0.3,4.793);
\node at (0.5,5.1) {{\footnotesize \faExclamation}};


\pic[gray!70] at (0,2) (myman) {man};
\node[below=1mm of myman-foot] {{\footnotesize spreader}};
\draw [fill=red,red] (0.1,2.39) rectangle (0.3,2.793);
\node at (0.5,3.1) {{\footnotesize \faComment}};


\pic[gray!70] at (0,0) (myman) {man};
\node[below=1mm of myman-foot] {{\footnotesize stifler}};
\draw [fill=blue!30!black,blue!30!black] (0.1,0.39) rectangle (0.3,0.793);
\node at (0.5,1.1) {{\footnotesize \faComment}};
\draw[ultra thick] (0.3,1.2) -- (0.7,1);
\draw[thick,white] (0.25,1.2) -- (0.65,1);
\end{tikzpicture}

};


\pic[gray!70] at (0,8) (myman) {man};
\draw [fill=red,red] (0.1,8.39) rectangle (0.3,8.793);
\pic[gray!70] at (0.5,8) (myman) {man};
\draw [fill=cyan!70,cyan!70] (0.6,8.39) rectangle (0.8,8.793);

\pic[gray!70] at (3,8) (myman) {man};
\draw [fill=red,red] (3.1,8.39) rectangle (3.3,8.793);
\pic[gray!70] at (3.5,8) (myman) {man};
\draw [fill=blue!20,blue!20] (3.6,8.39) rectangle (3.8,8.793);
\node at (4,9) {\tiny $1$};

\draw[->,gray!70,thick] (1.35,8.5) -- (2.5,8.5);
\draw [->,dashed] (0.8,9.15) to [bend left=45] (3.6,9.15);


\pic[gray!70] at (0,6) (myman) {man};
\draw [fill=red,red] (0.1,6.39) rectangle (0.3,6.793);
\pic[gray!70] at (0.5,6) (myman) {man};
\draw [fill=blue!20,blue!20] (0.6,6.39) rectangle (0.8,6.793);
\node at (0.95,7) {{\tiny $i$}};

\pic[gray!70] at (3,6) (myman) {man};
\draw [fill=red,red] (3.1,6.39) rectangle (3.3,6.793);
\pic[gray!70] at (3.5,6) (myman) {man};
\draw [fill=blue!20,blue!20] (3.6,6.39) rectangle (3.8,6.793);

\draw[->,gray!70,thick] (1.35,6.5) -- (2.5,6.5);
\draw [->,dashed] (0.8,7.15) to [bend left=45] (3.6,7.15);
\node at (4.25,7) {\tiny $(i+1)$};


\pic[gray!70] at (0,4) (myman) {man};
\draw [fill=red,red] (0.1,4.39) rectangle (0.3,4.793);
\pic[gray!70] at (0.5,4) (myman) {man};
\draw [fill=blue!20,blue!20] (0.6,4.39) rectangle (0.8,4.793);
\node at (1.25,5) {{\tiny $(k-1)$}};

\pic[gray!70] at (3,4) (myman) {man};
\draw [fill=red,red] (3.1,4.39) rectangle (3.3,4.793);
\pic[gray!70] at (3.5,4) (myman) {man};
\draw [fill=red,red] (3.6,4.39) rectangle (3.8,4.793);

\draw[->,gray!70,thick] (1.35,4.5) -- (2.5,4.5);
\draw [->,dashed] (0.8,5.15) to [bend left=45] (3.6,5.15);


\pic[gray!70] at (0,2) (myman) {man};
\draw [fill=red,red] (0.1,2.39) rectangle (0.3,2.793);
\pic[gray!70] at (0.5,2) (myman) {man};
\draw [fill=red,red] (0.6,2.39) rectangle (0.8,2.793);

\pic[gray!70] at (3,2) (myman) {man};
\draw [fill=blue!30!black,blue!30!black] (3.1,2.39) rectangle (3.3,2.793);
\pic[gray!70] at (3.5,2) (myman) {man};
\draw [fill=red,red] (3.6,2.39) rectangle (3.8,2.793);

\draw[->,gray!70,thick] (1.35,2.5) -- (2.5,2.5);
\draw [->,dashed] (0.3,3.15) to [bend left=45] (3.1,3.15);


\pic[gray!70] at (0,0) (myman) {man};
\draw [fill=red,red] (0.1,0.39) rectangle (0.3,0.793);
\pic[gray!70] at (0.5,0) (myman) {man};
\draw [fill=blue!30!black,blue!30!black] (0.6,0.39) rectangle (0.8,0.793);

\pic[gray!70] at (3,0) (myman) {man};
\draw [fill=blue!30!black,blue!30!black] (3.1,0.39) rectangle (3.3,0.793);
\pic[gray!70] at (3.5,0) (myman) {man};
\draw [fill=blue!30!black,blue!30!black] (3.6,0.39) rectangle (3.8,0.793);

\draw[->,gray!70,thick] (1.35,0.5) -- (2.5,0.5);
\draw [->,dashed] (0.3,1.15) to [bend left=45] (3.1,1.15);

\end{tikzpicture}

\end{center}
\caption{Illustration of all possible transitions between the different classes of individuals in the $k$-spreading Maki--Thompson model. The population is subdivided into ignorants, $i$-aware individuals for $i\in\{1,\ldots k-1\}$, spreaders and stiflers. The different transitions in the model involve interactions between a spreader and another individual. The resulting changes are presented by dashed arrows.}\label{FIG:transitions}
\end{figure}
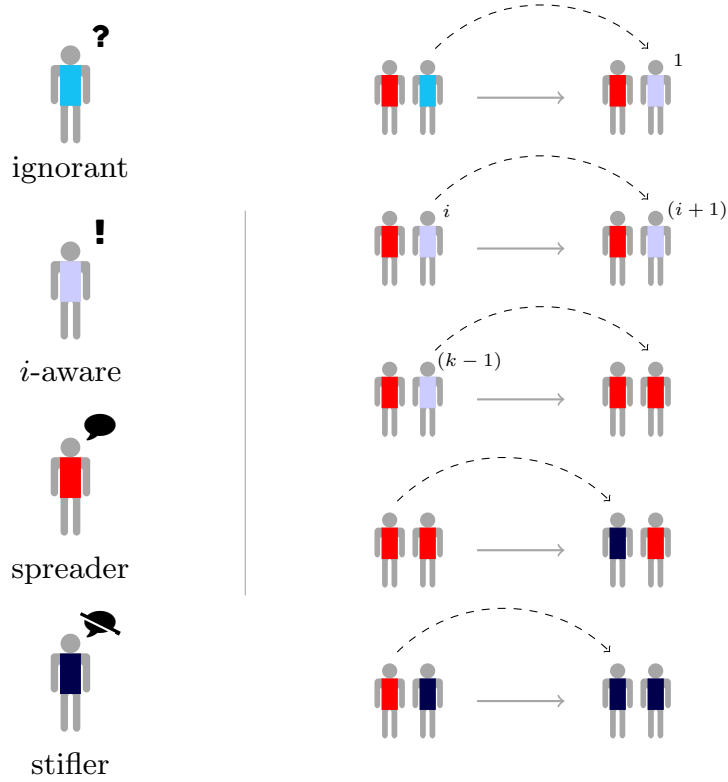

\subsection{The remaining proportion of individuals never hearing the rumor}

\smallskip
Our purpose is to investigate the remaining proportion of individuals in the different classes of the population when the rumor outbreak ends. Notice that such event occurs when there are no more spreaders in the population. In other words, if we consider the absorption time of the process defined as
\[ \tau^{(N)}:=\inf\{t\geq 0: Y^{(N)}(t)=0\}, \]
then we shall characterize the asymptotic behavior of the random variables $X^{N}\left(\tau^{(N)}\right)/N$, $Y_{i}^{N}\left(\tau^{(N)}\right)/N$ for $i\in \{1,\ldots,k-1\}$ and $Z^{N}\left(\tau^{(N)}\right)/N$, as $N \to \infty$. 
Part of our approach (see Section \ref{S:proofLLN} for details) is the identification of a limit dynamical system suitably defined, whose solution is partially expressed in terms of a transcendental function~$f$. Let $k\in\mathbb{N}$, $x_0\in (0,1]$ and $y_0, y_{i,0} \in [0,1]$, $i\in\{1,\ldots,k-1\}$, be fixed constants. 
In what follows, let $y_{0,0} := x_0$, and consider the function $f:(0,x_{0}]\to \mathbb{R}$ defined, for each $x\in(0,x_0]$, by
\begin{equation}\label{func:louca}
	f(x)= y_0+\rho(0) -\frac{x}{x_0} \sum_{r=0}^{k-1}\frac{\rho(r)}{r!} \left\{\ln \left(\frac{x_0}{x}\right)\right\}^r   -\ln \left(\frac{x_0}{x}\right),
\end{equation}
where
\begin{equation}\label{eq:rho}
\rho(r):=\sum_{j=0}^{k-r-1}(k-j-r+1)\, y_{j,0}
\end{equation}
for $r\in\{0,1,\ldots, k-1\}$. 
The function $f$ is the key for localizing the asymptotic remaining proportion of people never hearing the rumor, which appears as one of its zeros in $(0,x_0]$. Notice that $f$ is a continuous function on $(0,x_0]$ for which
\[ \lim_{x\to 0^{+}}f(x)=-\infty
\quad \text{and} \quad
f(x_0)=y_0 \geq 0. \]
Hence, we obtain the following result.

\begin{proposition}\label{prop:function}
Let $k\in\mathbb{N}$, $x_0\in (0,1]$ and $y_{i,0}, y_0 \in [0,1]$, for $i\in\{1,\ldots,k-1\}$, be fixed constants. If $f$ is the function defined by \eqref{func:louca}, then $f$ has at least one zero in the interval $(0,x_0]$.
\end{proposition}

The principal significance of Proposition \ref{prop:function} is that it allows the following definition.

\begin{definition}
Let $k\in\mathbb{N}$, $x_0\in (0,1]$ and $y_0, y_{i,0} \in [0,1]$, for $i\in\{1,\ldots,k-1\}$, be fixed constants. 
Let $f$ be the function defined by \eqref{func:louca}. Define
\begin{equation}\label{eq:xinfty}
x_{\infty}:=x_{\infty}(k,x_{0},y_{1,0},\ldots,y_{k-1,0},y_0)=\sup\{x\in (0,x_0]:f(x)<0\}.
\end{equation} 
\end{definition}

\begin{theorem}
\label{T: LLN}
Consider the $k$-spreading Maki--Thompson model with initial configuration $x_0, y_{1,0},\ldots, y_{k-1,0}, y_0$. Then
\begin{equation*}
\lim_{N\to \infty}\frac{{X}^{N}(\tau^{(N)})}{N}=x_\infty \;\; \text{a.s.} \quad \text{and}
\quad \lim_{N\to \infty}\frac{{Y_i}^{N}(\tau^{(N)})}{N}=y_{i,\infty} \;\; \text{a.s.},
\end{equation*}
where $x_{\infty}$ is given by \eqref{eq:xinfty} and
\begin{equation*}
y_{i,\infty}:= \frac{x_{\infty}}{x_0} \sum_{r=0}^{i} y_{i-r,0}\, \frac{\left\{\ln (x_0 / x_{\infty})\right\}^r}{r!}, 
\end{equation*}
for $i\in \{1,\ldots,k-1\}$.
\end{theorem}

\begin{figure}
\begin{center}
\includegraphics[width=8cm]{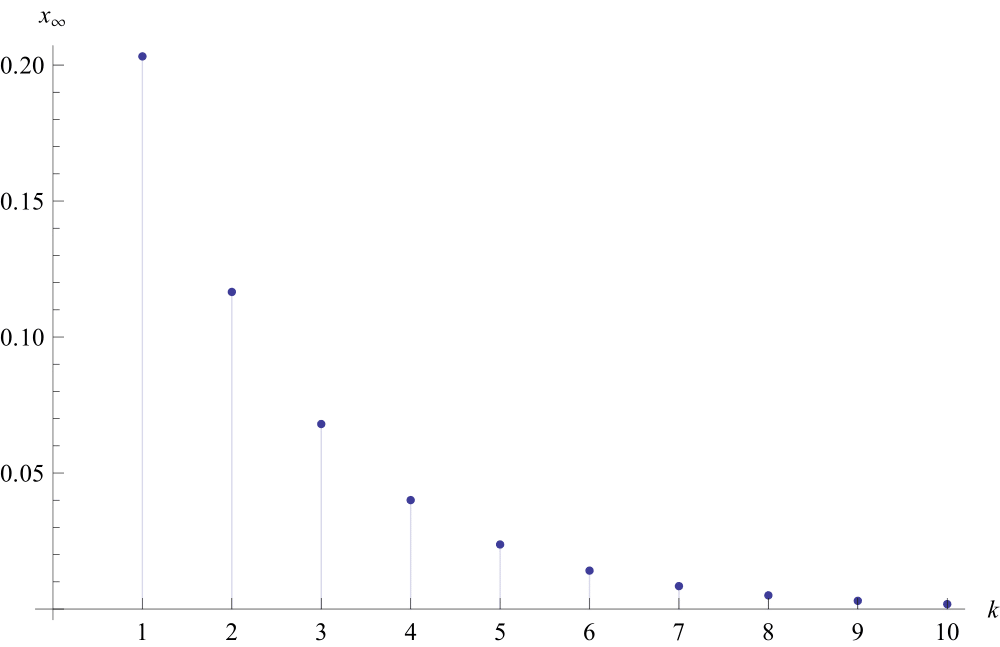}
\caption{Graph of $x_{\infty}$ as a function of $k$, for the standard initial configuration.}
\label{fig:xinfty_k}
\end{center}
\end{figure}

\begin{remark}
Theorem \ref{T: LLN} implies that $\lim_{N\to \infty}{Z}^{N}(\tau^{(N)})/N=z_{\infty}$ a.s., where $z_{\infty}:= 1- \left(x_{\infty} + \sum_{i=1}^{k-1} y_{i,\infty}\right)$. 
\end{remark}

The proof of Theorem \ref{T: LLN} is given in Section~\ref{S:proofLLN}. Now we present a result that yields information about the function $f$ and its zeros. Although this is a technical result we include it here because it is of mathematical interest in itself. Given a real function~$g$ and an interval $I$, let $\nzeros{g}{I}$ denote the number of zeros of $g$ in $I$.
For $k \geq 1$ an integer, let
\[ \gamma(k, t) = \int_{0}^{t} u^{k-1} \, e^{-u} \, du, \, t \geq 0, \]
be the \textit{lower incomplete gamma function}.

\begin{theorem}
\label{T: Prop-f}
\begin{itemize}
\item[(i)] Let $f$ be the function given by \eqref{func:louca}, for the $k$-spreading Maki--Thompson model with the standard initial configuration. Then

\begin{equation}
\label{F: f-CIP}
f(x) = \frac{(1+k+\ln x) \, \gamma(k,-\ln x) - x \, (-\ln x)^k}{(k-1)!}, \, x \in \intfzu.
\end{equation}
In addition, $\nzeros{f}{\intfzu} = 2$.
The two zeros of $f$ are $x_{\infty} < 1$ and $x_{0} = 1$.

\item[(ii)] Suppose that $x_{0} < 1$ and $y_{i,0} = 0 $ for every $i\in\{1,\dots,k-1\}$.
Then $\nzeros{f}{\intfzx} \in \{1, 3\}$.

\end{itemize}
\end{theorem}

For the standard initial configuration, the dependence of the limiting fraction $x_{\infty}$ of ignorants on $k$ is illustrated in Figure~\ref{fig:xinfty_k}.
As one would expect, increasing the value of $k$ pushes the value of $x_{\infty}$ towards zero.
Figure~\ref{fig:xinfty_num} shows some cases of the graph of the function~$f$, for $k=3$ and depending on the values of the initial condition.
We present the proof of Theorem~\ref{T: Prop-f} in Section~\ref{S:Proof f}.

\begin{figure}[!htbp]
\begin{center}
\includegraphics[width=4.5cm]{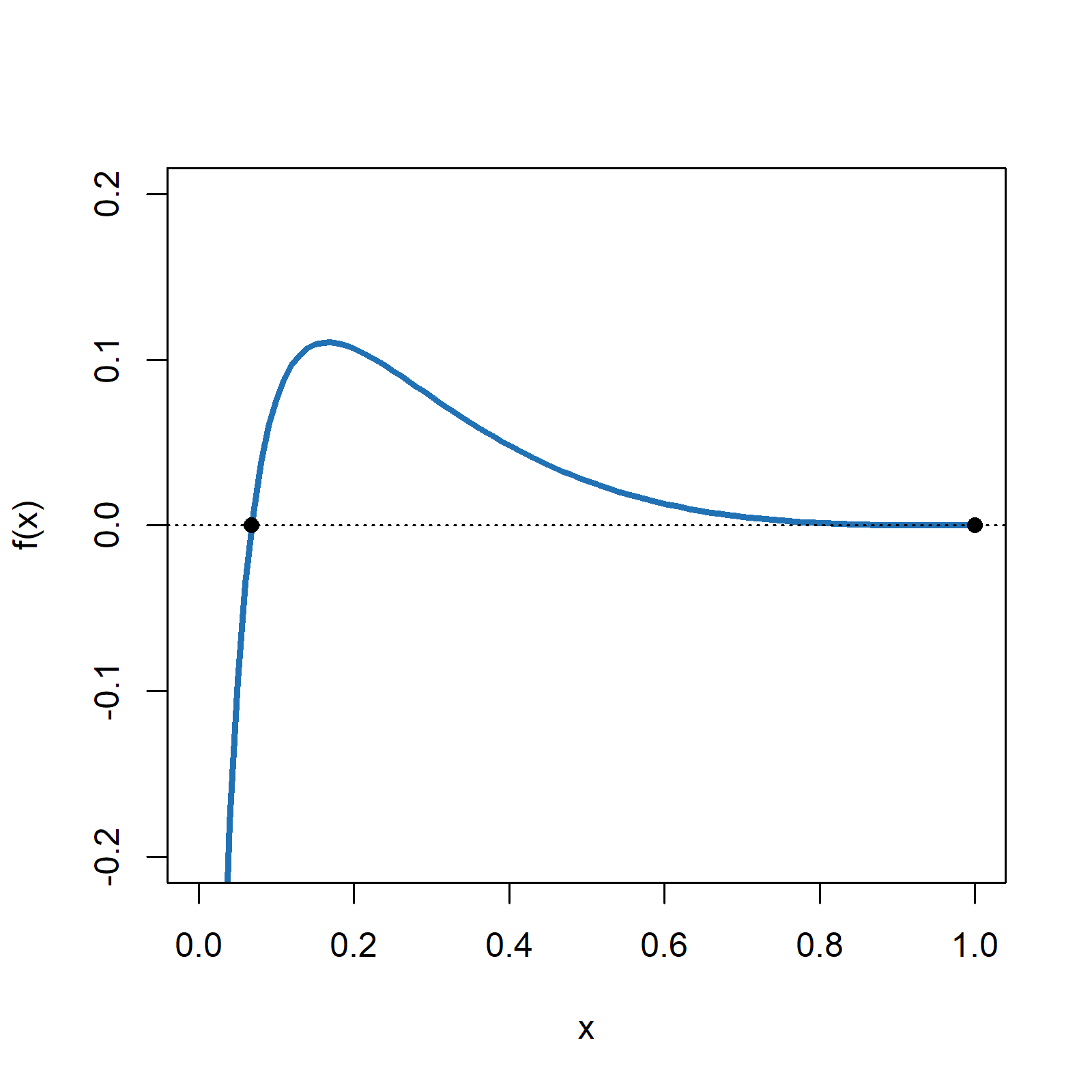}\qquad\includegraphics[width=4.5cm]{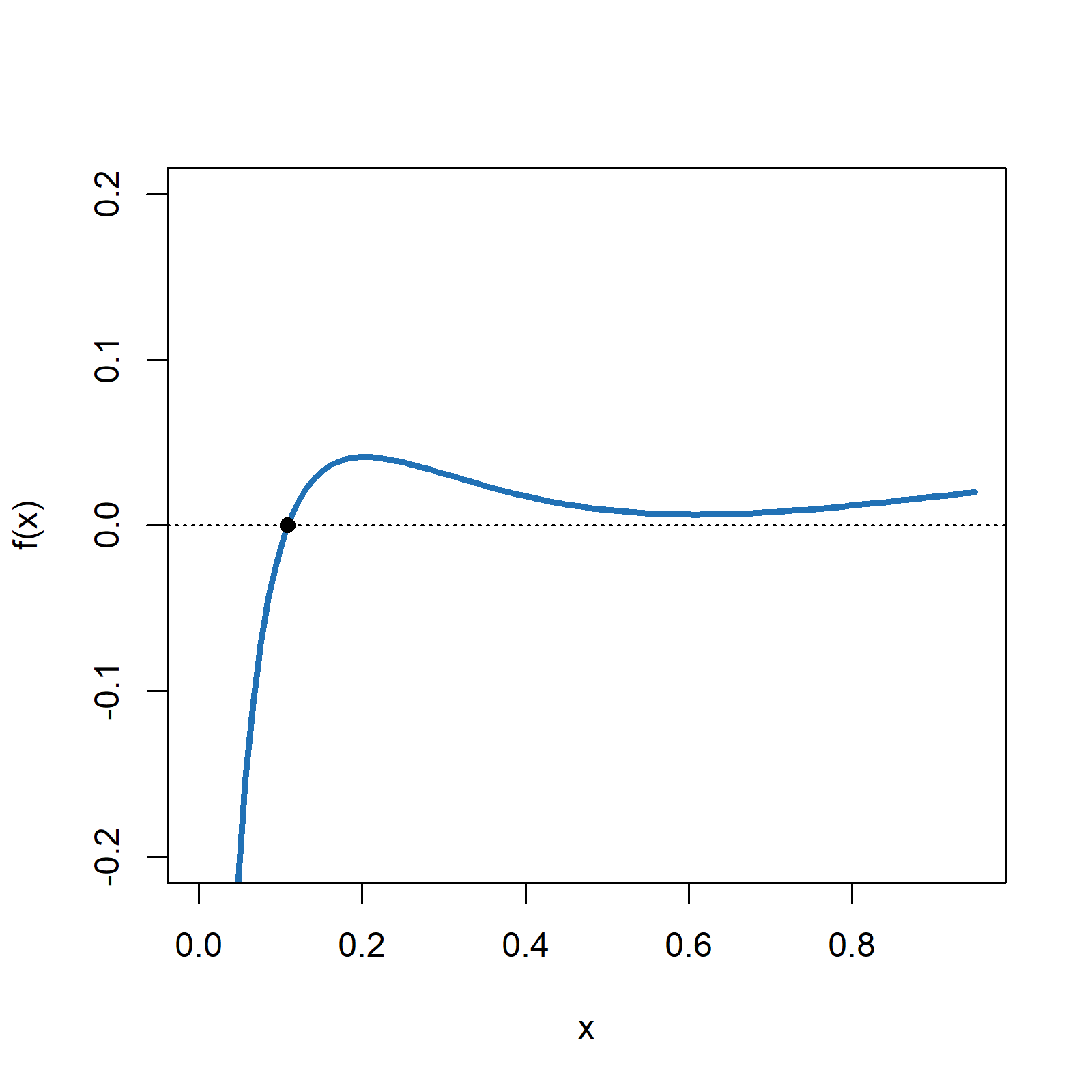}\qquad\includegraphics[width=4.5cm]{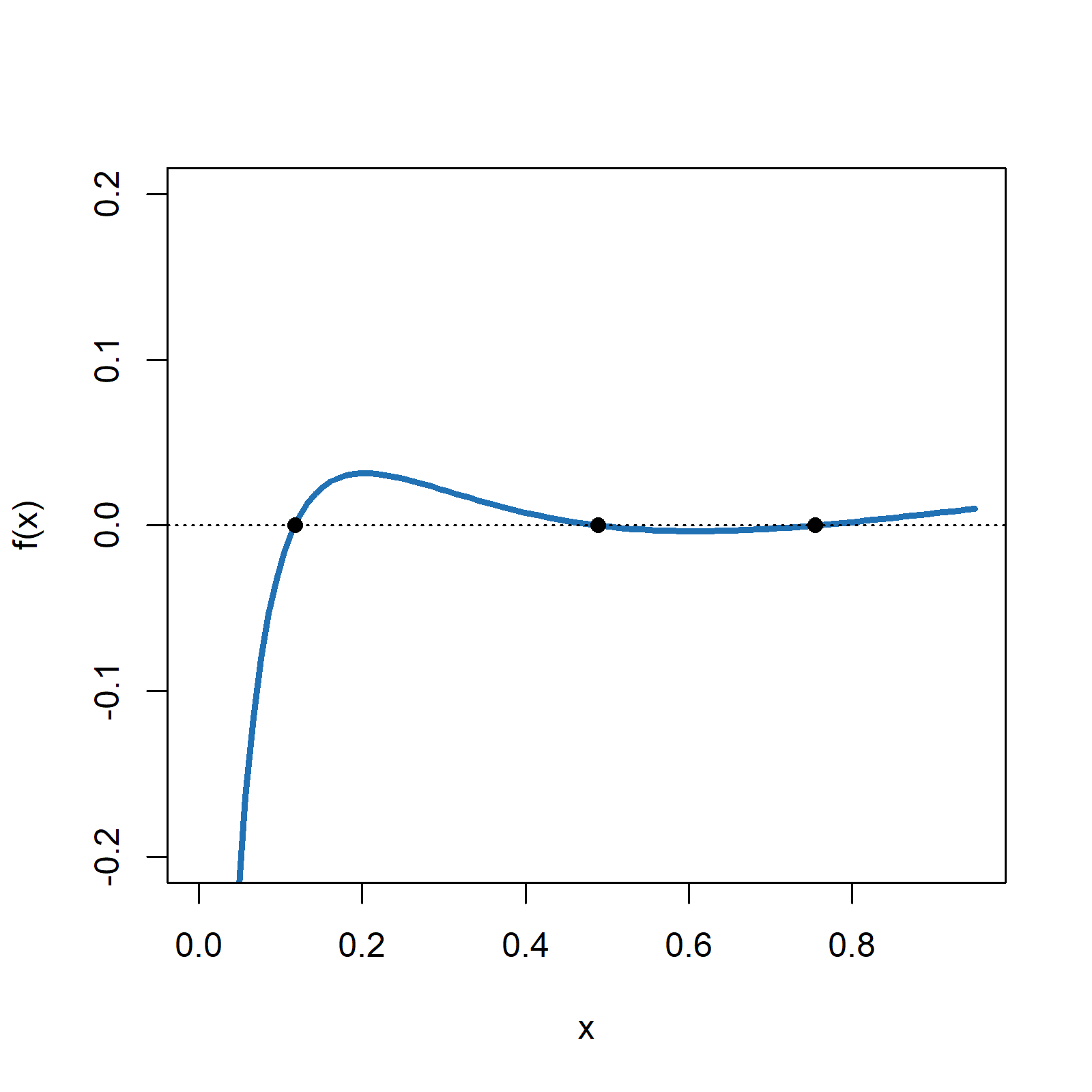}
\caption{Function $f$ for $k=3$. We consider the cases $x_0 = 1$ (standard initial configuration), $x_0 = 0.95, y_0 = 0.02, y_{1,0} = y_{2,0} = 0$ and $x_0 = 0.95, y_0 = 0.01, y_{1,0} = y_{2,0} = 0$, respectively, from the left to the right side. The zeros of the function $f$ are represented by black dots.}
\label{fig:xinfty_num}
\end{center}
\end{figure}

\subsection{On the random fluctuations between the ultimate proportions and their asymptotic values}

Theorem \ref{T: LLN} allows us to describe the asymptotic proportions of ignorants, $i$-aware individuals and stiflers during a given outbreak. As we will see later, the proof is accomplished by applying well-known convergence results from the theory of density dependent Markov chains. 
For a deeper study of this subject, we refer the reader to \cite[Chapter 11]{MPCC}. 
See also \cite[Chapter 5]{draief} for more details of this theory, with applications to epidemic-like processes. 

The main idea is to prove that, as $N$ goes to infinity, the entire trajectories of a suitable coupled version for the Markov chain, rescaled by $N$, converges to a solution of a tractable system of differential equations. It is worth pointing out that the same theory proves to be extremely useful to understand how the random fluctuations between the ultimate proportions and their asymptotic values behave as $N \to \infty$. Indeed, an  analysis similar to that used in \cite{Grejo/Rodriguez/2019,kurtz,lebensztayn/machado/rodriguez/2011a,lebensztayn/machado/rodriguez/2011b} shows that we can obtain a Central Limit Theorem for the $k$-spreading Maki--Thompson model. Under certain conditions, if  $x_{\infty}, y_{1,\infty}, \dots, y_{k - 1,\infty}$ are the values given by Theorem \ref{T: LLN}, then one should be able to prove that
\begin{equation}\label{EQ:geralTCL}
 \sqrt{N} \left(\frac{X^{N}(\tau^{(N)})}{N} - x_\infty, \frac{Y_{1}^{N}(\tau^{(N)})}{N} - y_{1,\infty}, \ldots, \frac{Y_{k-1}^{N}(\tau^{(N)})}{N} - y_{k-1,\infty}   \right) \Rightarrow \dNormal_k(0, \Sigma)
 \end{equation}
as $N \to \infty$, where $ \Rightarrow $ denotes convergence in distribution, and $\dNormal_k(0, \Sigma)$ is a $k$-variate normal distribution with mean zero and whose covariance matrix $\Sigma_{k \times k}$ may be written in terms of the constants $k$, $x_{\infty}$, $y_{1,\infty},\dots, y_{k - 1,\infty}$. A general result like \eqref{EQ:geralTCL} is useful and interesting; however, the method of proof involves many tedious computations. 
For simplicity, we state a Central Limit Theorem for the $k$-spreading Maki--Thompson model only for $k\in\{2,3\}$ and $x_0=1$, but we emphasize that the same arguments may be adapted for less restrictive assumptions.

\begin{theorem}
\label{T: CLT-2}
Consider the $2$-spreading Maki--Thompson model with the standard initial configuration. In this case,
\[ \xf \approx 0.116586 \quad \text{and} \quad
\yuf \approx 0.250558. \]
In addition, 
\[ \sqrt{N} \left(\frac{X^{N}(\tau^{(N)})}{N} - x_\infty, \frac{Y_{1}^{N}(\tau^{(N)})}{N} - y_{1,\infty} \right) \Rightarrow \dNormal_2(0, \Sigma)
\quad \text{as } N \to \infty, \]
where the elements of the covariance matrix $\Sigma$ are given by
\begin{equation}
\label{F: Cov Matrix-2}
{\allowdisplaybreaks
\begin{aligned}
\Sigma_{11} &= \frac{\xf \left(-4 \xf^3+\xf (3-4 \yuf)+4 \yuf^2-5
   \yuf+1\right)}{(\xf+2 \yuf-1)^2} \approx 0.179404,\\[0.1cm]
\Sigma_{12} &= \frac{6 \xf^4-3 \xf^3+\xf^2 (4 \yuf-3)+\xf (5-6 \yuf)
   \yuf-\yuf^2}{(\xf+2 \yuf-1)^2} \approx 0.0585937,\\[0.1cm]
\Sigma_{22} &= \frac{-9 \xf^5+9 \xf^4-3 \xf^3 \yuf+\xf^2 \yuf (9
   \yuf-7)}{\xf (\xf+2 \yuf-1)^2}+{}\\[0.1cm]
&\phantom{=}{}+\frac{\xf \yuf (\yuf+1)-\yuf^3}{\xf (\xf+2 \yuf-1)^2} \approx 0.288678.
\end{aligned}}%
\end{equation}
\end{theorem}

The density of the limiting bivariate normal distribution and the corresponding contour plot are depicted in Figure~\ref{Fig: Cont}. We observe that, as proved by Sudbury~\cite{Sudbury} and Watson~\cite{Watson}, for the standard ($1$-spreading) Maki--Thompson model, the ultimate proportion of ignorants and the variance of the asymptotic normal distribution in the CLT are $0.203188$ and $0.272736$, respectively. As an additional remark we point out that Daley and Kendall \cite{DK} provided an heuristic derivation of the mean and variance of an asymptotic CLT of  $0.203188$  and  $0.3106724$, which differs from the figures for the CLT of the Maki--Thompson model.

\begin{figure}[!htbp]
\centering
\includegraphics{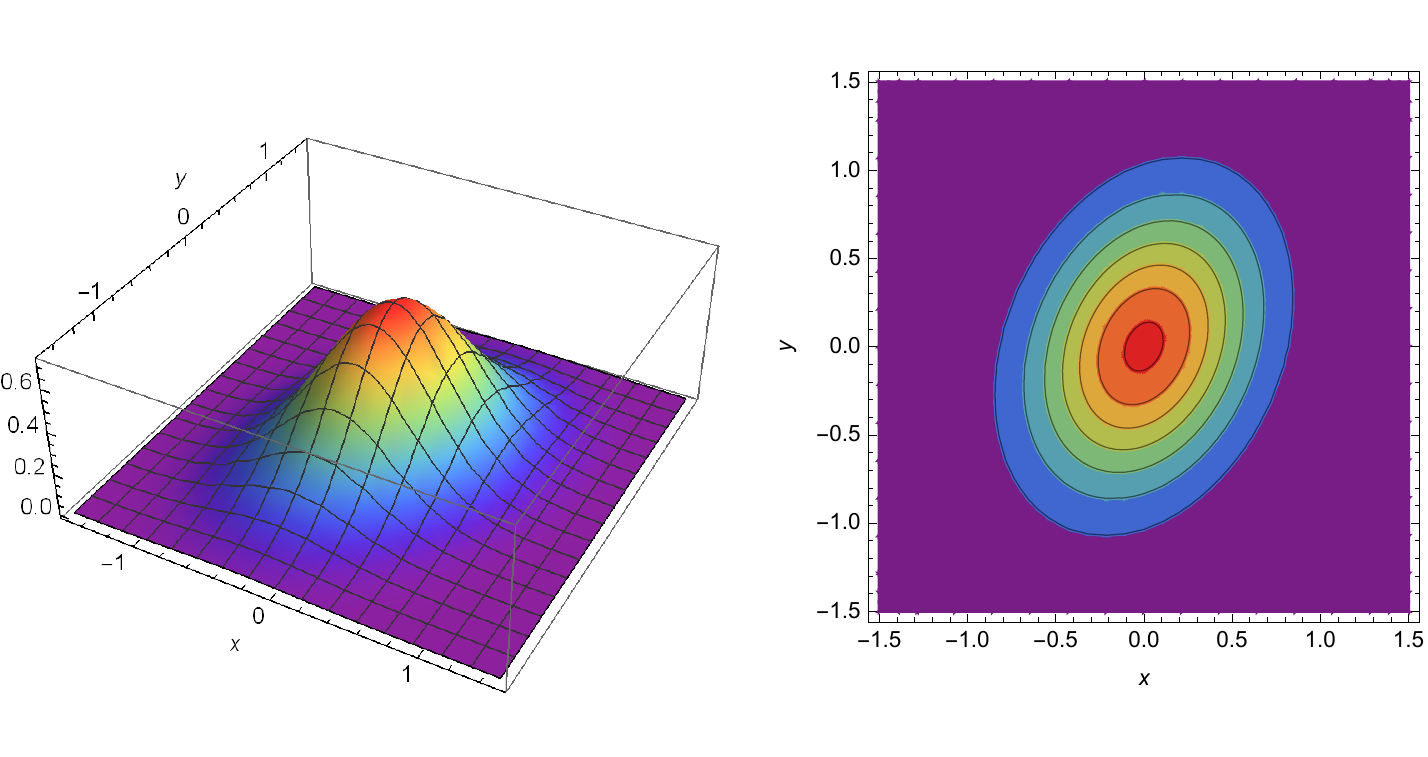}
\caption{Density and contour plot of the asymptotic bivariate normal density in Theorem~\ref{T: CLT-2}.}
\label{Fig: Cont}
\end{figure}

\begin{theorem}
\label{T: CLT-3}
Consider the $3$-spreading Maki--Thompson model with the standard initial configuration. In this case,
\[ \xf \approx 0.0680169, \quad
\yuf \approx 0.182829 \quad \text{and} \quad
\ydf \approx 0.245723. \]
Furthermore, 
\[ \sqrt{N} \left(\frac{X^{N}(\tau^{(N)})}{N} - x_\infty, 
\frac{Y_{1}^{N}(\tau^{(N)})}{N} - y_{1,\infty},
\frac{Y_{2}^{N}(\tau^{(N)})}{N} - y_{2,\infty} \right) \Rightarrow \dNormal_3(0, \Sigma)
\quad \text{as } N \to \infty, \]
where
\begin{equation}
\label{F: Cov Matrix-3}
\Sigma \approx 
\begin{pmatrix}
0.111645 & 0.0690173 & 0.0279058 \\
0.0690173 & 0.286895 & 0.0303917 \\
0.0279058 & 0.0303917 & 0.226601 \\
\end{pmatrix}.
\end{equation}
\end{theorem}

We present only the numerical values of the elements of $\Sigma$, because the expressions in terms of $\xf$, $\yuf$ and $\ydf$ are too long to write out explicitly.
The proofs of Theorems~\ref{T: CLT-2} and \ref{T: CLT-3} are presented in Section~\ref{S:Proofs CLT}.

\section{Proof of Theorem \ref{T: LLN}}\label{S:proofLLN}

From now on, we use the notation $\{\mathcal{R}(t)\}_{t\in[0,\infty)}$ for the $k$-spreading Maki--Thompson rumor model. The main idea of the proof is to define a new process $\{\tilde{\mathcal{R}}(t)\}_{t\in[0,\infty)}$ with the same transitions as $\{\mathcal{R}(t)\}_{t\in[0,\infty)}$ such  that they end at the same point. The interest in the comparison with this new process is that it allows us to apply well-known convergence results of density dependent Markov chains. We first prove that, as $N \to \infty$, the entire trajectories of $\{\tilde{\mathcal{R}}(t)\}_{t\in[0,\infty)}$, rescaled by $N$, have as limit a set of differential equations that we can study. Then the asymptotic proportions of interest are obtained from  analysis of this dynamical system. 
Thorough expositions of the theory of density dependent Markov chains are given in~\cite[Chapter 11]{MPCC} and \cite[Chapter 5]{draief}.
In what follows, we develop a similar approach to those in \cite{kurtz,Grejo/Rodriguez/2019,lebensztayn/machado/rodriguez/2011a,lebensztayn/machado/rodriguez/2011b}.

\subsection{The time-changed process and the limit dynamical system}
\label{SS: TCP}

Consider the $k$-spreading Maki--Thompson rumor model $\{\mathcal{R}^{(N)}(t)\}_{t\in[0,\infty)}$, and define the absorption time of the process by
\begin{equation*}
\tau^{(N)}=\inf\{t\geq 0: Y^{(N)}(t)=0\}.
\end{equation*}
Since the event $\{Y^{(N)}(t)=0\}$ implies that there are no spreaders in the population at time $t$, $\tau^{(N)}$ represents the time when the process ends. We define for $0\leq t \leq \int_0^{\infty} Y^{(N)}(s) ds$,
\begin{equation*}
\gamma^{(N)}(t):= \inf \left\{u\in[0,\tau^{(N)}): \int_0^u Y^{(N)}(s) ds>t\right\},
\end{equation*}
and note that 
\[ \int_0^{\gamma^{(N)}(t)} Y^{(N)}(s) ds=t. \]
Let $\tilde{\mathcal{R}}^{(N)}(t)=(\tilde{X}^{(N)}(t),\tilde{Y}^{(N)}_1(t), \ldots,\tilde{Y}^{(N)}_{k-1}(t), \tilde{Y}^{(N)}(t)):=\mathcal{R}^{(N)}(\gamma^{(N)}(t))$. Thus defined, the time-changed process $\{\tilde{\mathcal{R}}^{(N)}(t)\}_{t\in[0,\infty)}$ is a continuous-time Markov chain with the same transitions as $\{\mathcal{R}^{(N)}(t)\}_{t\in[0,\infty)}$. Namely,

\begin{equation}\label{taxas}
{\allowdisplaybreaks
\begin{array}{rccl}
&\text{increment} \quad &\text{rate} &\\[0.3cm]
\ell_0:=&- {\bf e}_1 + {\bf e}_2 \quad &\tilde{X}, &\\[0.3cm]
\ell_i:=&- {\bf e}_{i+1} + {\bf e}_{i+2} \quad &\tilde{Y}_i, &i\in\{1,2,\ldots,k-1\} \\[0.1cm]
\ell_k:=&- {\bf e}_{k+1} \quad &(N+1-\tilde{X}-\displaystyle\sum_{i=1}^{k-1}\tilde{Y}_i),& 
\end{array}}%
\end{equation}
where $\{{\bf e}_1, {\bf e}_2,\ldots,{\bf e}_{k+1}\}$ is the natural basis of the $(k+1)$-dimensional Euclidean space. Because this new Markov chain has the same transitions as the process $\{\mathcal{R}^{(N)}(t)\}_{t\in[0,\infty)}$, if these two processes start at the same point, they will also be absorbed at the same point. Formally, if $\tilde{\tau}^{(N)}:=\inf\{t\geq 0: \tilde{Y}^{(N)}(t)=0\}$, then
\begin{equation*}
\tilde{\mathcal{R}}^{(N)}\left(\tilde{\tau}^{(N)}\right)=\mathcal{R}^{(N)}\left(\tau^{(N)}\right).
\end{equation*}

Therefore, to prove Theorem \ref{T: LLN}, it is enough to prove limit theorems for the coordinates of $\tilde{\mathcal{R}}^{(N)}\left(\tilde{\tau}^{(N)}\right)/N$. We do this by noting that the new process is a density dependent Markov chain. That is, if we consider the functions
\begin{equation}
\label{F: Beta}
\begin{aligned}
\beta_{\ell_0}(x,y_1,\ldots,y_{k-1}, y)&=x,\\[0.1cm]
\beta_{\ell_i}(x,y_1,\ldots,y_{k-1}, y)&=y_i \text{ for }i\in\{1,\ldots,k-1\},\\[0.1cm]
\beta_{\ell_k}(x,y_1,\ldots,y_{k-1}, y)&=1-x-\displaystyle\sum_{i=1}^{k-1}y_i,
\end{aligned}
\end{equation}
the rates in \eqref{taxas} can be written as $N\beta_{\ell_i}(\tilde{X}/N,\tilde{Y}_1/N,\dots,\tilde{Y}_{k-1}/N,\tilde{Y})+O(1/N)$ for $i\in\{0,1,\dots,k\}$. 
Now we apply Theorem~11.2.1 from \cite{MPCC} to $\{\tilde{\mathcal{R}}^{(N)}(t)\}_{t\in[0,\infty)}$, which gives the characterization of a system of ordinary differential equations for which the scaled stochastic system converges almost surely, on bounded time intervals. So, following \cite{MPCC}, we identify the drift function associated to $\{\tilde{\mathcal{R}}^{(N)}(t)\}_{t\in[0,\infty)}$ as the function given by
\[ F(x,y_1,\ldots,y_{k-1}, y) = \sum_{i = 0}^{k} \ell_i \, \beta_{\ell_i}(x,y_1,\ldots,y_{k-1}, y). \]
Thus, if $v(t)=(x(t),y_{1}(t),\ldots,y_{k-1}(t), y(t))$ for $t \in [0,\infty)$ is such that  
\[ v(t):=v(0)+ \int_{0}^{t}F(v(s))\,ds, \quad t\geq 0, \]
and $v(0)=(x_0,y_{1,0},\ldots,y_{k-1,0}, y_0)$, then Theorem~11.2.1 from \cite{MPCC} guarantees that for every $t\geq 0$,
\begin{equation}\label{EQ:limitStoch_det}
\lim_{N\to \infty} \sup_{s\leq t} \left\|\frac{\tilde{\mathcal{R}}^{(N)}(s)}{N} - v(s)\right\| =0, \text{ a.s.}
\end{equation}
Consequently, the scaled stochastic system converges almost surely on bounded time intervals to the solution of the following limit dynamical system:
\begin{equation}
\begin{cases}\label{EQ:limitsystemODE}
x^{\prime}(t) = - x(t), \\[0.3cm]
y_1^{\prime}(t) = x(t)-y_1(t), \\[0.3cm]
y_i^{\prime}(t) = y_{i-1}(t)-y_i(t), \text{ for }i\in\{2,\ldots,k-1\} \\[0.1cm]
y^{\prime}(t) = y_{k-1}(t)-\left(1-x(t)-\displaystyle\sum_{i=1}^{k-1}y_i(t)\right), \\[0.3cm]
x(0)=x_0, y_i(0)=y_{i,0} \text{ for } i\in\{1,\ldots,k-1\}, y(0)=y_0.
\end{cases}
\end{equation}

\begin{remark}
Recall that the classical Maki--Thompson model is obtained from the $k$-spreading Maki--Thompson model by setting $k = 1$ and considering the standard initial configuration.
In this case, the system~\eqref{EQ:limitsystemODE} reduces to
\begin{equation}\label{EQ:limitMTS}
\left\{
\begin{array}{l}
x^{\prime}(t) = - x(t),\\[0.2cm]
y^{\prime}(t) = 2x(t)-1,\\[0.2cm]
x(0)=1, y(0)=0.
\end{array}\right.
\end{equation}
For the system~\eqref{EQ:limitMTS}, we may deduce directly the trajectory in the $(x, y)$-plane.
Indeed, \eqref{EQ:limitMTS} yields
\[ \frac{\partial y}{\partial x} = -2 + \frac{1}{x}, \, y(1) = 0, \]
so that
\[ y = f(x) = 2 \, (1-x) + \ln x, \, x \in \intfzu. \]
This is the function $f$ given in Equation~\eqref{F: f-CIP} with $k=1$ (see Theorem~\ref{T: Prop-f}).
Notice also that
\[ y = 0 \iff x = 1 \text{ or } x = \xf \approx 0.203188. \]
Summarizing, a fundamental point in the proof of Theorem~\ref{T: LLN} is to define the time-changed coupled version $\{\tilde{\mathcal{R}}(t)\}_{t\in[0,\infty)}$ of the original processes $\{\mathcal{R}(t)\}_{t\in[0,\infty)}$, both stochastic systems having the same embedded chain.
This allows us to apply the machinery of the density dependent population process and to obtain the dynamical system~\eqref{EQ:limitsystemODE}, whose solution we can compute.
\end{remark}

\subsection{The solution of the dynamical system}
\label{SS: Solution}

A simple computation shows that $x(t)=x_0 e^{-t}$. Since $y_1^{\prime}(t) + y_1(t) = x(t)$ with
$y_1(0)=y_{1,0}$ is a linear differential equation of first order, it can be solved by standard methods. It is not difficult to show that $y_1(t)=e^{-t} \left(x_0 t +y_{1,0} \right)$ is the only solution to this equation. The same technique applied now to the equation $y_i^{\prime}(t) + y_i(t) = y_{i-1}$ with $y_i(0)=y_{i,0}$, yields 
\begin{equation*}
y_i(t) =  e^{-t} \displaystyle \sum_{r=0}^i y_{i-r,0} \frac{t^r}{r!}, \text{ for }i\in\{2,\ldots,k-1\}.
\end{equation*}

Now, the differential equation for $y(t)$ reads
\begin{equation*}
y^{\prime}(t) = e^{-t} \displaystyle \sum_{r=0}^{k-1} y_{k-1-r,0} \frac{t^r}{r!} + x_0 e^{-t} + \sum_{i=1}^{k-1} \left(e^{-t} \displaystyle \sum_{r=0}^i y_{i-r,0} \frac{t^r}{r!}\right) - 1.
\end{equation*}
Observe that $y_i(t) =  e^{-t} P_i(t)$ where 
\[ P_i(t):=\displaystyle \sum_{r=0}^i y_{i-r,0} \frac{t^r}{r!}. \]
Define $P_i^{(s)}(t) := {d^s P_i(t)}/{dt^s}$. Since
\begin{align*}
P_i^{(1)}(t) &= \displaystyle \sum_{j=1}^i y_{i-j,0} \frac{t^{j-1}}{(j-1)!}\\
&= \displaystyle \sum_{j=1}^i y_{(i-1)-(j-1),0} \frac{t^{j-1}}{(j-1)!}\\
&= \displaystyle \sum_{l=0}^{i-1} y_{(i-1)-l,0} \frac{t^l}{l!}\\
&= P_{i-1}(t),
\end{align*}
we conclude that
\begin{equation}\label{s-derivada}
P_i^{(s)}(t) = P_{i-s}(t).
\end{equation}
Thus,
\begin{align*}
\displaystyle \int y_i(t) dt &=  \displaystyle \int e^{-t} P_i(t) dt \\
&= -e^{-t} P_i(t) + \displaystyle \int e^{-t} P_i^{(1)} dt \\
&= -e^{-t} P_i(t) + \displaystyle \int e^{-t} P_{i-1} dt,
\end{align*}
where the last line is obtained from~\eqref{s-derivada}. 
Using a recursive argument yields
\begin{align}\label{integral-ys}
\int y_i(t) dt 
&= -e^{-t} \displaystyle \sum_{s=0}^{i-1} P_{i-s}(t) - x_0 e^{-t} \nonumber \\
&= - \sum_{s=0}^{i-1}y_{i-s}(t)  - x(t)\nonumber\\
&= - \sum_{s=1}^{i}y_s(t) -x(t).
\end{align}



But $y'(t)= \displaystyle \sum_{i=1}^{k-1} y_i(t) + y_{k-1}(t) + x(t) - 1$, so using \eqref{integral-ys}, it follows that
\begin{align}\label{yt}
y(t) &= \int y'(t) dt \nonumber\\
&= \sum_{i=1}^{k-1} \int y_i(t) + \int y_{k-1}(t) + \int x(t)-t,\nonumber\\
&= -\sum_{i=1}^{k-1} \left(\sum_{s=1}^{i}y_s(t)+x(t)\right)-\sum_{s=1}^{k-1} y_s(t)-2x(t)-t + c\nonumber\\
&= -\sum_{i=1}^{k-1}\sum_{s=1}^{i}y_s(t)-\sum_{s=1}^{k-1} y_s(t)-(k+1) x(t)-t + c \nonumber\\
&= -\sum_{s=1}^{k-1} (k-s+1) y_s(t) - (k+1) x(t) - t + c,
\end{align}
where $c$ is a constant that depends on the initial condition.
By setting $t = 0$ (recall that $x(0) = x_0 = y_{0,0}$), we get
\begin{equation}
\label{c}
c = y_0 + \rho(0),
\end{equation}
where $\rho(0)$ is given by \eqref{eq:rho}. Thus, replacing \eqref{c} in \eqref{yt}, we obtain 
\begin{equation}
\label{yt-final}
y(t) = y_0 + \rho(0) - \sum_{s=1}^{k-1} (k-s+1) y_s(t) - (k+1) x(t) - t.
\end{equation}

For the sake of simplicity, in our analysis, we shall write $y(t)$ as a function of $x(t)$, that is, we prove that $y(t) = f(x(t))$, where $f$ is given by~\eqref{func:louca}.
Now let $x:=x(t)=x_0 e^{-t}$, so that
\begin{equation}
\label{eq:t}
t=\ln (x_0/x),
\end{equation}
and we may express $y_i(t)$ as a function of $x$ by
\begin{equation}
\label{yi-con-x}
y_i(t) = \frac{x}{x_0}\sum_{r=0}^{i}y_{i-r,0}\frac{(\ln (x_0/x))^r}{r!}, \, i \in \{1,\dots,k-1\}.
\end{equation}
Consequently, substituting from \eqref{eq:t} and \eqref{yi-con-x} into \eqref{yt-final} gives
{\allowdisplaybreaks
\begin{align*}
y(t) &= y_0 + \rho(0) - \frac{x}{x_0} \sum_{s=0}^{k-1} \sum_{r=0}^{s} (k-s+1) y_{s-r,0} \frac{(\ln (x_0/x))^r}{r!} - \ln (x_0/x)\\
&= y_0 + \rho(0) - \frac{x}{x_0} \sum_{r=0}^{k-1} \frac{(\ln (x_0/x))^r}{r!} \sum_{s=r}^{k-1} (k-s+1) y_{s-r,0} - \ln (x_0/x)\\
&= y_0 + \rho(0) - \ln (x_0/x) - \frac{x}{x_0} \sum_{r=0}^{k-1} \frac{\rho(r)}{r!} (\ln (x_0/x))^r 
= f(x(t)).
\end{align*}}%
Therefore, the solution of system \eqref{EQ:limitsystemODE} is given by
\begin{equation}
\begin{cases}
\label{EQ:sollimitsystemODE}
x(t) = x_0 e^{-t},\\
y_i(t) = \dfrac{x(t)}{x_0} \displaystyle\sum_{r=0}^{i} \dfrac{y_{i-r,0} \, t^r}{r!}
\quad \text{for} \quad i \in \{1,\dots,k-1\},\\
y(t) = f(x(t)).
\end{cases}
\end{equation}

\subsection{Final steps to prove Theorem \ref{T: LLN}}

Consider the $k$-spreading Maki--Thompson model $\left(\mathcal{R}^{(N)}(t)\right)_{t\geq 0}$, its time-changed coupled version $(\tilde{\mathcal{R}}^{(N)}(t))_{t\geq 0}$, and the deterministic system $(v(t))_{t\geq 0}$ whose coordinates are given by \eqref{EQ:sollimitsystemODE}. By \eqref{EQ:limitStoch_det} we have that for every $t\geq 0$,
\begin{equation}\label{EQ:XlimitStoch_det}
\lim_{N\to \infty} \sup_{s\leq t} \left|\frac{\tilde{X}^{(N)}(s)}{N} - x(s)\right| =0, \text{ a.s.}
\end{equation}
and
\begin{equation}\label{EQ:YlimitStoch_det}
\lim_{N\to \infty} \sup_{s\leq t} \left|\frac{\tilde{Y}^{(N)}(s)}{N} - y(s)\right| =0, \text{ a.s.}
\end{equation}
Thus, if we define $\tau_{\infty}:=\inf\{t\geq 0: y(t)\leq 0\}$, we deduce from \eqref{EQ:YlimitStoch_det} that
\begin{equation}\label{EQ:tau}
\lim_{N\to \infty} \tilde{\tau}^{(N)} = \tau_{\infty}, \text{ a.s. }
\end{equation}
Now recall that $y(t)=f(x(t))$, so $\tau_{\infty}=\inf\{t\geq 0: f(x(t))\leq 0\}$.
In addition, $x(t)$ is a decreasing function, hence it establishes a one-to-one correspondence between $[0, \infty)$ and $(0, x_0]$. 
The definition of $x_\infty$ implies that $x_{\infty} = x(\tau_{\infty})$.
Putting all together, the last remark, \eqref{EQ:XlimitStoch_det} and \eqref{EQ:tau} imply that 
\[ \lim_{N\to \infty} \frac{X^{(N)}(\tau^{N})}{N} = \lim_{N\to \infty} \frac{\tilde X^{(N)}(\tilde \tau^{N})}{N} = x_{\infty} \text{ a.s.}, \]
and the proof of Theorem \ref{T: LLN} is complete.

\section{Proof of Theorem \ref{T: Prop-f}}
\label{S:Proof f}

For $t \geq 0$, define
\[ \phi(t) = f(x_0 \, e^{-t}) = y_0 + \rho(0) - t 
- e^{-t} \sum_{r=0}^{k-1}\frac{\rho(r)}{r!} \, t^r. \]
Of course, $\phi(t)$ is equal to $y(t) = f(x(t))$ given in~\eqref{EQ:sollimitsystemODE}, but we prefer to use different notation, to keep the sections independent.
Notice that there is a one-to-one correspondence between $\intfzi$ and $\intfzx$, so that 
$f(x)= \phi(\ln(x_0/x))$.

We apply a probabilistic argument to show formula~\eqref{F: f-CIP}. As usual, for an event $A$, we write $E(R; A) = E(R \, I_A)$. We use the following result.

\begin{lemma}
Let $R$ be a random variable with Poisson distribution with parameter $t > 0$.
For every integer $k \geq 1$,
\begin{subequations}
\label{F: Poisson}
\begin{equation}
\label{F: Poisson-P}
P(R \geq k) = \frac{\gamma(k, t)}{(k-1)!}, \quad \text{and}
\end{equation}

\begin{equation}
\label{F: Poisson-E}
E(R; R > k) = t \, P(R \geq k).
\end{equation}
\end{subequations}
\end{lemma}

\begin{proof}
Formula~\eqref{F: Poisson-P} is the well-known relationship between Poisson and Gamma distributions.
To prove~\eqref{F: Poisson-E}, notice that
\[ \frac{P(R = r + 1)}{P(R = r)} = \frac{t}{r + 1}. \]
From this, it follows that
\[ E(R; R > k) = \sum_{r=k}^{\infty} (r + 1) \, P(R = r + 1)
= t \, P(R \geq k).\]
\end{proof}

\begin{proof}[Proof of Equation~\eqref{F: f-CIP}]
We may express $\phi(t)$ as
\begin{equation}
\label{F: phi-P}
\phi(t) = y_0 + \rho(0) - E(R) - E(\rho(R); R < k).
\end{equation}
Under the standard initial configuration, $x_0 = 1$, $y_0 = 0$ and
$\rho(r) = k-r+1$ for $r = 0, \dots, k-1$, whence
\[ E(\rho(R); R < k) = E(k-R+1; R < k) = 
(k+1) \, P(R < k) - E(R; R < k). \]
Therefore, from~\eqref{F: Poisson} and \eqref{F: phi-P}, we obtain
\begin{align*}
\phi(t) &= (k+1) \, P(R \geq k) - E(R; R \geq k) \\[0.1cm]
&= (k+1-t) \, P(R \geq k) - k \, P(R = k) \\[0.1cm]
&= \frac{(k+1-t) \, \gamma(k,t) - e^{-t} \, t^k}{(k-1)!}.
\end{align*}
Since $f(x) = \phi(-\ln x)$, we arrive at~\eqref{F: f-CIP}.
\end{proof}

Now we prove the assertions about the number of zeros of the function $f$ that are stated in parts {(i)} and {(ii)} of Theorem~\ref{T: Prop-f}.
The following result is useful for this purpose.

\begin{theorem}[P\'{o}lya and Szeg\H{o}\cite{PS}, p.~41, Problems V.38 and V.40]
\label{T: PS}
Let $\psi(t) = \sum_{n = 0}^{\infty} a_n \, t^n$ be a power series whose radius of convergence is $\infty$.
Denote by $C$ the number of changes of sign in the sequence of coefficients of $\psi$.
If $C$ is finite, $C - \nzeros{\psi}{\intazi}$ is a nonnegative even number.
\end{theorem}

For $t \geq 0$, we define
\[ \psi(t) = (y_0 + \rho(0) - t) \, e^t 
- \sum_{r=0}^{k-1} \frac{\rho(r)}{r!} \, t^r. \]
Clearly, $\nzeros{f}{\intazx} = \nzeros{\phi}{\intazi} = \nzeros{\psi}{\intazi}$.
For $n \geq 0$, let $\psi^{(n)}$ denote the $n$th derivative of $\psi$.
Then, for $n = 0, \dots, k-1$,
\[ \psi^{(n)}(t) = (y_0 + \rho(0) - n - t) \, e^t 
- \sum_{r=n}^{k-1} \frac{\rho(r)}{(r-n)!} \, t^{r-n}, \]
whereas for $n \geq k$,
\[ \psi^{(n)}(t) = (y_0 + \rho(0) - n - t) \, e^t. \]
Hence, the coefficient of the general term in a power series expansion for $\psi$ around $0$ is given by
\[ a_n =
\left\{
\begin{array}{cl}
\dfrac{y_0 + \rho(0) - \rho(n) - n}{n!} &\text{if } 0 \leq n \leq k - 1, \\[0.4cm]
\dfrac{y_0 + \rho(0) - n}{n!} &\text{if } n \geq k.
\end{array}	\right. \]

\begin{proof}[Proof of Part (i)]
Since $\rho(r) = k-r+1$ for every $r = 0, \dots, k-1$ under the standard initial configuration, we obtain that
\begin{align*}
&a_n = 0 \, \text{ for } 0 \leq n \leq k - 1,\\[0.1cm]
&a_k = 1 / k! > 0,\\[0.1cm]
&a_{k + 1} = 0, \quad \text{and}\\[0.1cm]
&a_n < 0 \, \text{ for } n \geq k + 2.
\end{align*}
Therefore, $C = 1$, and thus, from Theorem~\ref{T: PS}, $\nzeros{f}{\intazu} = \nzeros{\psi}{\intazi} = 1$.
Consequently, $\nzeros{f}{\intfzu} = 2$.
Denoting by $\tau_{\infty}$ the unique positive zero of $\psi$, we conclude that the two zeros of $f$ are $x_{\infty} = e^{-\tau_{\infty}} < 1$ and $x_{0} = 1$.
\end{proof}

\begin{proof}[Proof of Part (ii)]
Assume that $x_{0} < 1$ and $y_{i,0} = 0 $ for every $i = 1,\dots,k-1$.
In this case, $\rho(r) = (k-r+1) \, x_0$ for $r = 0, \dots, k-1$, whence
\[ a_n =
\left\{
\begin{array}{cl}
\dfrac{y_0 - n \, (1 - x_0)}{n!} &\text{if } 0 \leq n \leq k - 1, \\[0.4cm]
\dfrac{y_0 + (k + 1) \, x_0 - n}{n!} &\text{if } n \geq k.
\end{array}	\right. \]
It is not difficult to prove that the sequence $\{a_n\}$ has the following properties:
\begin{enumerate}[label=(\alph*),itemsep=1ex]
\item $a_0 = y_0 \geq 0$.

\item If $a_n \leq 0$ for some $0 \leq n \leq k - 2$, then $a_{n + 1} < 0$.

\item $a_{n} < 0$ for every $n \geq k + 1$.

\end{enumerate}
Using Theorem~\ref{T: PS}, we have the following cases to consider:
\begin{enumerate}
\item[{\bf (1)}] $y_0 = 0$ and $x_0 \leq k/(k + 1)$:
Then $C = \nzeros{\psi}{\intazi} = 0$, so $x_0$ is the unique zero of the function~$f$ in $\intfzx$.

\item[{\bf (2)}] $y_0 = 0$ and $x_0 > k/(k + 1)$:
Here $C = 2$, hence $\nzeros{\psi}{\intazi} \in \{0, 2\}$ and $\nzeros{f}{\intfzx} \in \{1, 3\}$.

\item[{\bf (3)}] $y_0 > 0$ and $y_0 + (k + 1) \, x_0 \leq k$:
$C = \nzeros{\psi}{\intazi} = \nzeros{f}{\intfzx} = 1$, and $x_{\infty}$ is the unique zero of~$f$ in $\intfzx$.

\item[{\bf (4)}] $y_0 > 0$ and $y_0 + (k + 1) \, x_0 > k$:
$C \in \{1, 3\}$, thus $\nzeros{\psi}{\intazi} = \nzeros{f}{\intfzx} \in \{1, 3\}$.

\end{enumerate}

\smallskip
This finishes the proof of part (ii).
\end{proof}

\section{Proofs of Theorems~\ref{T: CLT-2} and \ref{T: CLT-3}}
\label{S:Proofs CLT}

\subsection{Proof of Theorem~\ref{T: CLT-2}}

Let $\{\mathcal{R}^{N}(t)\}_{t\in[0,\infty)}$ be the $2$-spreading Maki--Thompson model with the standard initial configuration.
For each transition $\ell_i$ of the process, we consider the corresponding $\beta_{\ell_i}$ function given by
\begin{equation}
\label{F: Rates TCP-2}
\begin{array}{cc}
\text{Increment} &\quad \text{Rate} \\[0.2cm]
\ell_0 = (-1, 1, 0) &\quad \beta_{\ell_0}\vetord = x, \\[0.3cm]
\ell_1 = (0, -1, 1) &\quad \beta_{\ell_1}\vetord = y_1,\\[0.3cm]
\ell_2 = (0, 0, -1) &\quad \beta_{\ell_2}\vetord = 1 - x - y_1.
\end{array}
\end{equation}
As explained in Section~\ref{SS: TCP} (see \eqref{F: Beta}), we may couple $\{\mathcal{R}^{N}(t)\}_{t\in[0,\infty)}$ with a density dependent population process $\{\tilde{\mathcal{R}}^{N}(t)\}_{t\in[0,\infty)}$ starting from the same initial configuration and whose transition scheme is described by~\eqref{F: Rates TCP-2}, in such a way that they have the same final values.

Towards proving Theorem~\ref{T: CLT-2}, we apply Theorem~11.4.1 of \cite{MPCC} to the process $\tilde{\mathcal{R}}^{N}$.
We adopt the notations used there; in our case, the Gaussian process $V$ defined on p.~458 is three-dimensional, and denoted by $\GV = (\GV_x, \GV_{y_1}, \GV_y)$.
The initial steps are presented for general $k$ and initial configuration in Sections~\ref{SS: TCP} and~\ref{SS: Solution}.
Moreover, some of the required computations were carried out with the help of symbolic mathematical software.
The drift function associated to $\tilde{\mathcal{R}}^{N}$ is given by
\[ F\vetord = \sum_{i = 0}^{2} \ell_i \, \beta_{\ell_i}\vetord 
= (-x, x - y_1, 2 y_1 + x - 1). \]
Consequently, the system of ordinary differential equations for which the scaled stochastic system converges almost surely on bounded time intervals is 
\begin{equation}
\label{F: ODE-2}
\begin{cases}
x^{\prime}(t) = -x(t), \\[0.1cm]
y_1^{\prime}(t) = x(t) - y_1(t), \\[0.1cm]
y^{\prime}(t) = 2 y_1(t) + x(t) - 1, \\[0.1cm]
x(0) = 1, y_1(0) = 0, y(0) = 0.
\end{cases}
\end{equation}
The solution of~\eqref{F: ODE-2} is $r(t) = (x(t), y_1(t), y(t))$, where
\begin{equation*}
x(t) = e^{-t}, \quad
y_1(t) = t \, e^{-t}, \quad
y(t) = f(x(t)),
\end{equation*}
and $f(x) = 3 (1 - x) + (2 x + 1) \ln x$.
We consider the function $\varphi\vetord = y$, so that
\begin{align*}
\tf &= \inf \{t \geq 0: y(t) \leq 0 \} \approx 2.14913, \\[0.1cm]
\xf &= x(\tf) \approx 0.116586, \quad \text{and} \\[0.1cm]
\yuf &= y_1(\tf) \approx 0.250558.
\end{align*}
Furthermore, we have that
\begin{equation}
\label{F: df}
\df = \nabla \varphi(r(\tf)) \cdot F(r(\tf)) = 2 \, \yuf + \xf - 1 \approx -0.382298 < 0.
\end{equation}
From Theorem~11.4.1 of \cite{MPCC}, we conclude that
\begin{equation}
\label{F: VP-2}
\sqrt{N} \left(\frac{X^{N}(\tau^{(N)})}{N} - x_\infty, 
\frac{Y_{1}^{N}(\tau^{(N)})}{N} - y_{1,\infty},
\frac{Y^{N}(\tau^{(N)})}{N} \right)
\end{equation}
converges in distribution as $N \to \infty$ to
\begin{equation*}
\GV(\tf) - \frac{\GV_{y}(\tf)}{\df} \, F(r(\tf)),
\end{equation*}
where $\df$ is defined in~\eqref{F: df}.
Let $\lf$ denote the covariance matrix of $\GV(\tf)$, and define
\[ B =
\begin{pmatrix}
1 & 0 & \xf / \df \\
0 & 1 & (-\xf + \yuf) / \df
\end{pmatrix}. \]
Hence, the random vector formed by the first two components of the vector in~\eqref{F: VP-2} converges in distribution as $N \to \infty$ to a mean zero bivariate normal distribution, whose covariance matrix is expressed by
\begin{equation}
\label{F: CM-2}
\Sigma = B \, \lf \, B^T,
\end{equation}
where $B^T$ is the transpose of $B$.
To finish the proof, it remains to explain the main steps in the computation of $\lf = \Cov(\GV(\tf), \GV(\tf))$.
First, the matrix of partial derivatives of the drift function $F$ and the matrix $G$ are given by
\begin{equation*}
{\allowdisplaybreaks
\begin{aligned}
\partial F\vetord &= 
\begin{pmatrix}
-1 & 0 & 0 \\
 1 & -1 & 0 \\
 1 & 2 & 0 
\end{pmatrix} \quad \text{ and} \\
G\vetord &= \sum_{i = 0}^{2} \ell_i \, \ell_i^T \, \beta_{\ell_i}\vetord =
\begin{pmatrix}
 x & -x & 0 \\
-x & x+y_1 & -y_1 \\
 0 & -y_1 & 1-x 
\end{pmatrix}.
\end{aligned}}%
\end{equation*}
Moreover, the solution $\Phi$ of the matrix equation
\[ \frac{\partial}{\partial t} \, \Phi(t, s) = \partial F(x(t), y_1(t), y(t)) \, \Phi(t, s),
\quad \Phi(s, s) = I_3, \]
where $I_3$ denotes the identity matrix of order $3$, is obtained as
\[ 
\Phi(t, s) =
\begin{pmatrix}
e^{s-t} & 0 & 0 \\
-(s-t) e^{s-t} & e^{s-t} & 0 \\
 3 + [2 (s-t)-3] e^{s-t} & 2 (1-e^{s-t}) & 1
\end{pmatrix}.
\]
Then, since the covariance matrix of the Gaussian process $\GV$ is
\begin{equation*}
\Cov(\GV(t), \GV(r)) = \int_0^{t \wedge r} \Phi(t, s) \, G(x(s), y_1(s), y(s)) \, {[\Phi(r, s)]}^T \, ds,
\end{equation*}
we obtain that $\lf$ has the following elements:
{\allowdisplaybreaks
\begin{align*}
\lf^{1,1} &= \xf (1 - \xf), \qquad \lf^{1,2} = -\xf \yuf, \qquad \lf^{1,3} = 3 \xf^2 + (2 \yuf - 3) \xf + \yuf = 0, \\[0.1cm]
\lf^{2,2} &= \yuf (1 - \yuf), \qquad \lf^{2,3} = \frac{\yuf
   \left(3 \xf^2 + (2 \yuf - 3) \xf + \yuf\right)}{\xf} = 0, \\[0.1cm]
\lf^{3,3} &= \frac{-9
   \xf^3 + (9 - 12 \yuf) \xf^2 + 2 (1 - 2 \yuf) \yuf \xf - 4
   \yuf^2 + \yuf}{\xf}. 
\end{align*}}%
Finally, using~\eqref{F: CM-2}, we get formula~\eqref{F: Cov Matrix-2}.\qed

\subsection{Proof of Theorem~\ref{T: CLT-3}}

We prove the Central Limit Theorem for the final outcome of the $3$-spreading Maki--Thompson model, starting from the standard initial configuration, by following a similar line of arguments as in the proof of Theorem~\ref{T: CLT-2}.
For convenience and ease of presentation, we summarize the main steps:
\begin{enumerate}
\item[{\bf (1)}] Transition scheme of the time-changed process $\{\tilde{\mathcal{R}}^{N}(t)\}_{t\in[0,\infty)}$:
\begin{equation*}
\begin{array}{cc}
\text{Increment} &\quad \text{Rate} \\[0.2cm]
\ell_0 = (-1, 1, 0, 0) &\quad \beta_{\ell_0}\vetort = x, \\[0.3cm]
\ell_1 = (0, -1, 1, 0) &\quad \beta_{\ell_1}\vetort = y_1,\\[0.3cm]
\ell_2 = (0, 0, -1, 1) &\quad \beta_{\ell_2}\vetort = y_2,\\[0.3cm]
\ell_3 = (0, 0, 0, -1) &\quad \beta_{\ell_3}\vetort = 1 - x - y_1 - y_2.
\end{array}
\end{equation*}

\item[{\bf (2)}] Basic definitions and computations:
{\allowdisplaybreaks
\begin{align*}
\GV &= (\GV_x, \GV_{y_1}, \GV_{y_2}, \GV_y) \quad \text{(Gaussian process)},\\[0.1cm]
F\vetort &= \sum_{i = 0}^{3} \ell_i \, \beta_{\ell_i}\vetort 
= (-x, x - y_1, y_1 - y_2, 2 y_2 + y_1 + x - 1),\\[0.1cm]
\varphi\vetort &= y,\\[0.1cm]
\partial F\vetort &= 
\begin{pmatrix}
-1 & 0  & 0  & 0 \\
 1 & -1 & 0  & 0 \\
 0 & 1  & -1 & 0 \\
 1 & 1  & 2  & 0 \\
\end{pmatrix},\\[0.1cm]
G\vetort &= \sum_{i = 0}^{3} \ell_i \, \ell_i^T \, \beta_{\ell_i}\vetort =
\begin{pmatrix}
 x & -x & 0 & 0 \\
-x & x+y_1 & -y_1 & 0 \\
 0 & -y_1 & y_1+y_2 & -y_2 \\
 0 & 0 & -y_2 & 1-x-y_1 \\
\end{pmatrix}.
\end{align*}}%

\item[{\bf (3)}] System of ordinary differential equations:
\begin{equation*}
\begin{cases}
x^{\prime}(t) = -x(t), \\[0.1cm]
y_1^{\prime}(t) = x(t) - y_1(t), \\[0.1cm]
y_2^{\prime}(t) = y_1(t) - y_2(t), \\[0.1cm]
y^{\prime}(t) = 2 y_2(t) + y_1(t) + x(t) - 1, \\[0.1cm]
x(0) = 1, y_1(0) = 0, y_2(0) = 0, y(0) = 0.
\end{cases}
\end{equation*}

\item[{\bf (4)}] Solution:
$r(t) = (x(t), y_1(t), y_2(t), y(t))$, where
\begin{equation*}
x(t) = e^{-t}, \quad
y_1(t) = t \, e^{-t}, \quad
y_2(t) = \frac{t^2 \, e^{-t}}{2}, \quad
y(t) = f(x(t)),
\end{equation*}
and $f(x) = 4 (1 - x) + (3 x + 1) \ln x - x \, \ln^2 x$.

\item[{\bf (5)}] Final values:
\begin{align*}
\tf &= \inf \{t \geq 0: y(t) \leq 0 \} \approx 2.688, \\[0.1cm]
\xf &= x(\tf) \approx 0.0680169, \\[0.1cm]
\yuf &= y_1(\tf) \approx 0.182829, \\[0.1cm]
\ydf &= y_2(\tf) \approx 0.245723, \quad \text{and} \\[0.1cm]
\df &= \nabla \varphi(r(\tf)) \cdot F(r(\tf)) = 2 \, \ydf + \yuf + \xf - 1 \approx -0.257709 < 0.
\end{align*}

\item Matrix $\Phi$:
\[ 
\Phi(t, s) =
\begin{pmatrix}
e^{s-t} & 0 & 0 & 0 \\
-e^{s-t} (s-t) & e^{s-t} & 0 & 0 \\
e^{s-t} (s-t)^2 / 2 & -e^{s-t} (s-t) & e^{s-t} & 0 \\
4-e^{s-t} \left(s^2-(2 t+3) s+t^2+3 t+4\right) & e^{s-t} (2
   (s-t)-3)+3 & 2 \left(1-e^{s-t}\right) & 1 \\
\end{pmatrix}.
\]

\item[{\bf (6)}] Computation of $\lf = \Cov(\GV(\tf), \GV(\tf))$:
Since the Gaussian process $\GV$ has covariance matrix
\begin{equation*}
\Cov(\GV(t), \GV(r)) = \int_0^{t \wedge r} \Phi(t, s) \, G(x(s), y_1(s), y_2(s), y(s)) \, {[\Phi(r, s)]}^T \, ds,
\end{equation*}
we get
\[ 
\lf \approx
\begin{pmatrix}
\phantom{-}0.0633906 & -0.0124355 & -0.0167133 & 0 \\
-0.0124355 & \phantom{-}0.149403\phantom{5} & -0.0449253 & 0 \\
-0.0167133 & -0.0449253 & \phantom{-}0.185343\phantom{3} & 0 \\
0 & 0 & 0 & 0.692721 \\
\end{pmatrix}.
\]

\item[{\bf (7)}] CLT obtained from Theorem~11.4.1 of \cite{MPCC}:
The random vector
\begin{equation}
\label{F: VP-3}
\sqrt{N} \left(\frac{X^{N}(\tau^{(N)})}{N} - x_\infty, 
\frac{Y_{1}^{N}(\tau^{(N)})}{N} - y_{1,\infty},
\frac{Y_{2}^{N}(\tau^{(N)})}{N} - y_{2,\infty},
\frac{Y^{N}(\tau^{(N)})}{N} \right)
\end{equation}
converges in distribution as $N \to \infty$ to
\begin{equation*}
\GV(\tf) - \frac{\GV_{y}(\tf)}{\df} \, F(r(\tf)).
\end{equation*}

\item[{\bf (8)}] Conclusion: The random vector formed by the first three components of the vector in~\eqref{F: VP-3} converges in distribution as $N \to \infty$ to a mean zero trivariate normal distribution, whose covariance matrix is given by
\( \Sigma = B \, \lf \, B^T \),
where 
\[ B =
\begin{pmatrix}
1 & 0 & 0 & \xf / \df \\
0 & 1 & 0 & (-\xf + \yuf) / \df \\
0 & 0 & 1 & (-\yuf + \ydf) / \df \\
\end{pmatrix}
\approx
\begin{pmatrix}
1 & 0 & 0 & -0.263929 \\
0 & 1 & 0 & -0.445513 \\
0 & 0 & 1 & -0.244048 \\
\end{pmatrix}. \]
Hence, we obtain formula~\eqref{F: Cov Matrix-3}.\qed
\end{enumerate}

\section{Conclusion}

In this work, we contribute to the theory of mathematical models for information spreading by formulating a generalization for the Maki--Thompson rumor model. In our model, we assume that each ignorant becomes a spreader only after hearing the rumor a number $k$ of times, for $k\in\mathbb{N}$. By letting $k=1$ we recover the original Maki--Thompson model so our results generalize those for the basic model. Our generalization is motivated by  recent experimental research, where it has been found that repeated presentation of uncertain statements increases validity judgments of those statements, see \cite{difonzo}. This is the well-known illusory truth effect, which is the tendency to believe false information to be correct after repeated exposure, see also \cite{hasher,polage}. Therefore, our model may be used as a mathematical model illustrating the impact of multiple-repetition in the transmission of information. Since our model is formulated as a multidimensional continuous-time Markov chain, we can establish limit theorems by proving convergence of its trajectory to a solution of a system of ordinary differential equations. The advantage of our approach lies in the fact that we obtain results that not only localize the asymptotic proportion of individuals in the different classes of the population, but also we provide a complete description of the random fluctuations between these values and those from the original stochastic process. Moreover, with more work, our approach may be extended to related models such as the Daley--Kendal.

\section*{Acknowledgments}
This work has been partially supported by FAPESP (Grants 2017/10555-0, 2018/22972-8), CAPES (under the Program MATH-AMSUD/CAPES 88881.197412/2018-01), and CNPq (Grants 304676/2016-0, 439422/2018-3). Part of this work was carried out during visits of P.M.R. at the UFABC and of A.R. at the UFPE. They are grateful to these universities for their hospitality and support. Special thanks are also due to the three anonymous reviewers for their helpful comments and suggestions.


\begin{thebibliography}{99}
\bibitem{EAP}
E. Agliari, A. Pachon, P. Rodriguez and F. Tavani, {\it Phase transition for the Maki--Thompson rumor model on a small-world network}, J. Stat. Phys., 169 n.4 (2017), pp. 846-875.
%
%
\bibitem{arruda}
G. Arruda, E. Cozzo, Y. Moreno, P. Rodriguez and F. Rodrigues, {\it A General Markov Chain Approach for Disease and Rumor Spreading in Complex Networks}, Journal of Complex Networks, 6 n.2 (2018), pp. 215-242.
%
%
\bibitem{raey}
C. F. Coletti, P. M. Rodr\'iguez and R. B. Schinazi, {\it A Spatial Stochastic Model for rumor Transmission}, J. Stat. Phys., 147 (2012), pp. 375-381.
%
%
\bibitem{DG}
D. J. Daley, J. Gani, Epidemic Modelling: An Introduction. Cambridge University Press, 1999.
%
\bibitem{DKNature}
D. J. Daley and D. G. Kendall, {\it Epidemics and rumours}, Nature, 204 (1964), p. 1118.
%
\bibitem{DK}
D. J. Daley and D. G. Kendall, {\it Stochastic rumours}, J. Inst. Math. Appl., 1 (1965), pp. 42--55. 
%
\bibitem{difonzo}
N. DiFonzo, J. W. Beckstead, N. Stupak and K. Walders, {\it Validity judgments of rumors heard multiple times: the shape of the truth effect}, Social Influence, 11 n.1 (2016), pp. 22-39.
%
\bibitem{draief}
M. Draief and L. Massouli\'e, {\it Epidemics and Rumours in Complex Networks, Cambridge University Press}, 2009.
%
\bibitem{MPCC} 
S. N. Ethier and T. G. Kurtz, \textit{Markov processes. Characterization and convergence}, Wiley Series in Probability and Mathematical Statistics, John Wiley \& Sons, New York, 1986.
%
%
\bibitem{Grejo/Rodriguez/2019}
C. Grejo and P. M. Rodriguez, {\it Asymptotic behavior for a modified Maki--Thompson model with directed inter-group interactions}, J. Math. Analysis and Applications, 480 (2019), p. 123402.
%
\bibitem{hasher}
L. Hasher, D. Goldstein and T. Toppino, \textit{Frequency and the conference of referential validity}, J. Verb. Learning Verb. Behav., 16 n.1 (1977), 107-112.
%
 \bibitem{speroto}
V. V. Junior, P. M. Rodriguez and A. Speroto, {\it The Maki--Thompson rumor model on infinite Cayley trees}, J. Stat. Phys., 181 (2020), pp. 1204-1217.
%
\bibitem{kurtz}
T. G. Kurtz, E. Lebensztayn, A. R. Leichsenring and F. P. Machado, {\it Limit theorems for an epidemic model on the complete graph}, ALEA. Latin American Journal of Probability and Mathematical Statistics, 4 (2008), pp. 45--55.
%
\bibitem{Lebensztayn-JMAA2015}
E. Lebensztayn, {\it A large deviations principle for the Maki--Thompson rumour model}, J. Math. Analysis and Applications, 432 (2015), pp. 142-155.
%
\bibitem{lebensztayn/machado/rodriguez/2011a}
E. Lebensztayn, F. P. Machado and P. M. Rodr\'iguez, {\it On the behaviour of a rumour process with random stifling}, Environ. Modell. Softw., 26 (2011), pp. 517-522.
%
\bibitem{lebensztayn/machado/rodriguez/2011b}
E. Lebensztayn, F. Machado and P. M. Rodr\'iguez, {\it Limit Theorems for a General Stochastic Rumour Model}, SIAM J. Appl. Math., 71 (2011), pp. 1476-1486.
%
\bibitem{EP}
E. Lebensztayn and P. M. Rodriguez, {\it A connection between a system of random walks and rumor transmission}, Physica A, 392 (2013), pp. 5793-5800.
%
\bibitem{maki-1973}
D. P. Maki, M. Thompson, {\it Mathematical Models and Applications. With emphasis on the social, life, and management sciences. Prentice-Hall, Englewood Cliffs}, N.~J., 1973.
%
\bibitem{MNP-PRE2004}
Y. Moreno, M. Nekovee and A. F. Pacheco, {\it Dynamics of rumour spreading in complex networks}, Phys. Rev. E, 69 (2004), p. 066130.
%
\bibitem{moreno-PhysA2007}
M. Nekovee, Y. Moreno, G. Bianconi and M. Marsili, {\it Theory of rumour spreading in complex social neworks}, Physica A, 374 (2007), pp. 457-470.
%
%
%
\bibitem{polage}
D. C. Polage, {\it Making up History: False Memories of Fake News Stories}, Eur. J. Psychol., 8 n.2 (2012), pp. 245-250.
%
\bibitem{PS}
G. P\'{o}lya and G. Szeg\H{o}, \textit{Problems and theorems in analysis II}, Classics in Mathematics, Springer-Verlag, Berlin, 1998.
%
\bibitem{Sudbury}
A. Sudbury, {\it The proportion of the population never hearing a rumour}, \emph{J. Appl. Probab.}, 22 (1985), pp. 443-446.
%
%
\bibitem{Watson}
R. Watson, {\it On the size of a rumour}, Stochastic Process. Appl., 27 (1988), pp. 141-149.
%
%
%

\end{thebibliography}

\end{document}